\documentclass[10pt,a4paper]{article} 
\usepackage[english]{babel}

\usepackage{commath}
\usepackage{graphicx}
\usepackage{enumitem}

\usepackage{float}
\usepackage{amsthm,amssymb,mathrsfs,amsmath}
\usepackage[a4paper,bindingoffset=0.2in,
left=0.5in,right=0.5in,
top=1in,bottom=1in,footskip=.25in]{geometry}
\date{\today}
\theoremstyle{definition}
\newtheorem{defn}{Definition}[section]
\newtheorem{theorem}{Theorem}[section]

\newtheorem{lemma}[theorem]{Lemma}
\newtheorem{prop}{Proposition}
\numberwithin{equation}{section}

\title{Region of existence of multiple solutions for a class of 4-point BVPs}
\author{Amit K. Verma$^a$\footnote{akverma@iitp.ac.in}, Nazia Urus$^b$\\\small{\it{$^{a,b}$Department of Mathematics, IIT Patna, Bihta, Patna 801103, (BR) India.}}}
\begin{document}
\maketitle
\begin{abstract}
The aim of this article is to prove the existence of solution and compute the region of existence for a class of 4-point BVPs defined as,
\begin{eqnarray*}
&&-u''(x)=\psi(x,u,u'), \quad 0<x<1,\\
&&u'(0)=\lambda_{1}u(\xi), \quad u'(1)=\lambda_{2} u(\eta),
\end{eqnarray*}
where $I=[0,1]$, $0<\xi\leq\eta<1$ and $\lambda_1,\lambda_2> 0$. The non linear source term $\psi\in C(I\times\mathbb{R}\times\mathbb{R},\mathbb{R})$ is one sided Lipschitz in $u$ with Lipschitz constant $L_1$ and Lipschitz in $u'$ with Lipschitz function $L_2(x)$, where $L_2(x):I\rightarrow \mathbb{R}^+ $ such that $ L_2(0)=0$ and $L_2'(x)\geq 0$. The novelty in this paper allows us to use simplest form of computational iteration and existence is achieved with a restriction that $L_2$ depends on $x$. We develop monotone iterative technique in well ordered and reverse ordered cases. We prove maximum anti-maximum principle under certain assumptions and use it to show the monotonic behavior of sequences of upper and lower solutions. The conditions derived in this paper are sufficient and have been verified for two examples. The method involves Newton's quasilinearization which involves a parameter $k$ which is equivalent to $\dfrac{\partial \psi}{\partial u}$. Our aim is to find a range of $k$ so that the iterative technique is convergent.
\end{abstract}
\section{Introduction}\label{section}
In the field of differential equations the concept of non-linear boundary value problems (NLBVPs) have great importance. In recent years, second and higher order NLBVPs have been studied in several areas to describe many physical, biological and chemical phenomena. To study the existence, multiplicity and positivity of solutions of non linear $m-$point BVPs various methods have been introduced.

These methods have also been used to obtain solutions of various kind of BVP such as Neuman, Dirichlet, mixed type BVP as well as Singular and Non singular BVP.\\

In 2010, to show the existence and multiplicity of positive solutions, Yang \cite{yang2010existence} used Krasnoselskii fixed point theorem and triple fixed point theorem for the following problem
\begin{eqnarray*}
&&u''(x)+\psi(x,u, u')=0, \quad 0<x<1,\\
&&u(0)=\alpha u(\eta) , \quad u(1)=\beta u(\xi).
\end{eqnarray*}
 Chen \cite{chen2006positive} provided existence of positive solutions for the following fourth order four-point nonlinear differential equation.
 \begin{eqnarray*}
&&u^4(x)+\psi(x,u)=0, \quad 0<x<1,\\
&&u(0)=u(1)=0 , \\
&& a u''(\xi_1)-b u'''(\xi_1)=0, \quad a u''(\xi_2)+b u'''(\xi_2)=0,
\end{eqnarray*}
where $0\leq \xi_1 \leq \xi_2\leq 1$ and $a, b,c,d\geq 0$ are constants. This equation is known as linear beam equation for $\psi(x,u)=a(x)g(u(x))$. He used  method of upper-lower solution and fixed point theorem to solve the above BVPs.

Zhai \cite{zhai2014properties} established the existence and uniqueness of positive solutions and used fixed point theorems of concave operators in partial ordering Banach spaces for a class of four-point BVPs of Caputo fractional differential equations for any given parameter. More works on fixed point theorem can be found in \cite{anderson1998multiple, bai2007positive,  bai2005positive, liu2005positive, webb2012existence}. Sun \cite{sun2011positive} studied a class of second order four point boundary value problems with $p$-Laplacian.

Chinni \cite{chinni2019four} studied existence, localization and multiplicity of positive solutions for four-point BVPs with singular $\phi$-Laplacian
\begin{eqnarray*}
&&-[\phi(u')]'=\psi(x,u, u'), \quad u(0)=\alpha u(\xi) , \quad u(T)=\beta u(\eta),
\end{eqnarray*}
where $\alpha,\beta \in[0,1),\quad 0<\xi<\eta<T $, $\psi:[0,T]\times \mathbb{R}^2\rightarrow \mathbb{R}  $ is continuous and $ \phi:(-a,a)\rightarrow \mathbb{R} ~(0<a<\infty) $ is an increasing homeomorphism and it is always solvable.

Palamides et al. \cite{palamides2012fourth} used continuum property of the solutions funnel (Kneser's Theorem) which is combined with the corresponding vector field for  fourth-order four-point BVPs. They investigated the existence of positive or a negative solution, although these problems do not always admit positive Green's function.

Some more works are done by using various methods such as Leray Schauder degree theory \cite{bai2007positive, ge2010existence, ma2004existence}, Shooting method \cite{taliaferro1979nonlinear}, Coincidence degree theory \cite{tang2016existence}, Topological degree method \cite {anderson1998multiple,bai2005positive,khan2006existence,liu2005positive,ma2004existence}, Upper-lower solution method \cite{khan2006existence,khan2007existence} and the list is not exhaustive. Bai at el. \cite{bai2007positive} studied second-order four-point BVPs
\begin{eqnarray*}
&&u''(x)+\lambda h(x)\psi(x,u)=0, \quad 0<x<1,\\
&&u'(0)=a u(\xi) , \quad u(1)=b u(\eta),
\end{eqnarray*}
where $0<\xi<\eta<1, ~ 0\leq a,b<1,$ and $h:[0,1]\rightarrow [0,\infty),~ \psi:[0,1]\times [0,\infty)\rightarrow [0,\infty) $
are continuous functions. The methods which are used to ensure the existence, non existence and multiplicity of positive solutions in a given range are fixed-point index theory, Leray-Schauder degree and the upper and lower solution. Shen at el. \cite{shen2011positive} established the existence of positive solution for a class of second-order four point problem
\begin{eqnarray*}
&&-u''(x)=\psi(x,u), \quad 0<x<1,\\
&&u(0)=\alpha u(\eta) , \quad u(1)=\beta u(\xi),
\end{eqnarray*}
where $\xi, \eta \in (0,1) $ under some resonance conditions for $\alpha$ and $\beta$ and they have used Leggett-Williams norm-type theorem.

The afore mentioned methods come with its own set of  advantages and disadvantages. The monotone iterative technique (MI-Technique) is an inspiring method \cite{cherpion2001constructive, zhang1995positive}  which gives ground for the theoretical and numerical existence and uniqueness of solution of non linear IVPs and $m$-point BVPs. Higher order $m$-point ODE and PDE are also studied by this technique.

MI-technique which is related to method of upper-lower solutions was first introduced by Picard \cite{picard1890memoire} in 1890. For Dirichlet BVPs he studied the existence of solution. In this method, for a class of linear problems, the monotone sequences of lower-upper solution  with initial guesses are constructed. Then by using the initial approximations, convergence of these monotone sequences are shown. Then it is shown that the solution of nonlinear BVPs lies between the convergent sequences of lower-upper solutions. For a comprehensive and detailed study, we suggest to read books \cite{VL1985,CDCPH2006}.

There are some special types of problems in which iterative technique is well grounded such as impulsive integro differential equations. For this problem Zhimin \cite{he2004monotone} investigated the existence and uniqueness of solution where boundary conditions are periodic. Other works which may be referred are such as Riemann-Liouville fractional differential equations \cite{wang2012monotone, cui2017monotone, ramirez2009monotone}, impulsive differential equations in Banach space \cite{chen2016double}, casual differential equation \cite{wang2016monotone},  Stieltjes integral boundary conditions \cite{sun2018monotone} etc.

Wei et al. \cite{wei2017monotone} studied the existence and uniqueness of slanted cantilever beam,
\begin{eqnarray*}
&&u^{4}(x)=\psi(x,u,u'), \quad 0\leq x\leq 1,\\
 &&u(0)=u'(0)=u''(1)=u'''(1)=0,
\end{eqnarray*}
where $\psi\in (I\times\mathbb{R}\times\mathbb{R},\mathbb{R})$ is continuous and $u'(x)$ is slope, reflecting the curving degree of the elastic beam. Above problem illustrates the static deformation of an elastic beam which has its right extreme freed and left extreme fixed. Here  Verma et al. \cite{verma2014existence} studied the existence of solution for three point BVPs which  arises in bridge design, and for this they considered the following problems
\begin{eqnarray*}
&&-u''(x)=\psi(x,u,u'), \quad 0<x<1,\\
&&u'(0)=0, \quad u'(1)=\delta u'(\eta),
\end{eqnarray*}
where $0<\eta<1$, $\delta> 0$ and $\psi\in (I\times\mathbb{R}^2,\mathbb{R})$ is continuous.

There are a lot of works in two, three and multi-point BVPs using this technique. In 1931 Dragoni \cite{dragoni1931ii} introduced MI-technique for two point Dirichlet BVPs when nonlinear term is derivative dependent. Cabada  et al. \cite{cabada2001monotone}, Cherpion et al. \cite{cherpion2001constructive}, also developed this technique for two point second order BVPs and studied the existence and approximation. Zhang \cite{zhang1995positive} proved necessary and sufficient conditions for existence of positive solution for the following  Dirichlet singular problem
\begin{eqnarray*}
&&-u''(x)=\psi(x,u), \quad 0<x<1,\\
&&u(0)=u(1)=0.
\end{eqnarray*}
As the exact solution for fractional differential equation can not be obtained easily, so we look for approximate solutions. For the approximation of solution various methods can be used but MI-Technique is an effective mechanism for both IVPs and BVPs related to fractional type differential equations. Cui \cite{cui2017monotone} used this technique to approximate maximal and minimal solutions and derived uniqueness result for non-linear Riemann-Liouville fractional differential equation.
\begin{eqnarray*}
&&D^{p}u(x)+\psi(x,u)=0, \quad 0<x<1,\\
&&u(0)=u'(1)=0, \quad u(1)=0,
\end{eqnarray*}
where $D^{p}$ is the standard Riemann-Liouville derivative and $p\in (2,3]$.

To deal with singular-nonlinear BVPs there are various methods such as Shooting method, Topological degree method and the method of lower-upper solutions but the method of upper-lower solution is very promising method \cite{zhang1995positive}. In \cite{pandey2008existence, Pandey2008existence1, pandey2009note, Pandey2010, PANDEY2010AMC} the uniqueness and existence of solutions for a class of singular and doubly singular two point BVPs have been established. Also the region of multiple solutions have been determined.

Verma et al. \cite{verma2015note}, Singh et al. \cite{singh2013picard, singh2013monotone, msakv2017} and  Li et al. \cite{li2008existence} developed monotone iterative technique for three point BVPs and studied the existence. In \cite{singh2013monotone}, Singh et al. considered three-point BVPs of the following type
\begin{eqnarray*}
&&-u''(x)=\psi(x,u,u'), \quad 0<x<1,\\
&&u'(0)=0, \quad u'(1)=\lambda_{2} u(\eta).
\end{eqnarray*}
Then they developed MI-Technique and derived some existence results.

MI-Technique has also been done for four point BVPs. Ge et al. \cite{ge2016multiplicity} studied multiplicity of solution for four point BVPs via the variational approach and  MI-Technique. Zhang et al. \cite{zhang2006upper} developed the upper and lower solution method and the monotone iterative technique and obtained some new existence results for the following fourth order four point BVPs
\begin{eqnarray*}
&&-u^{4}(x)=\psi(x,u,u''), \quad 0<x<1,\\
&&u(0)=u(1)=0, \\
&& a u''(\xi_1)-b u'''(\xi_1)=0, \quad c u''(\xi_2)+d u'''(\xi_2)=0.
\end{eqnarray*}
Recently, Verma et al. \cite{verma2019monotone} proved existence of solution for the following class of four point BVPs
\begin{eqnarray*}
&&-u''(x)=\psi(x,u,u'), \quad 0<x<1,\\
&&u'(0)=0, \quad u(1)=\delta_{1} u(\eta_{1})+\delta_{2}u(\eta_{2}),
\end{eqnarray*}
where $\psi$ is continuous, $I=[0,1]$, $\eta_1<\eta_2 \in (0,1) $ and $\delta_1,\delta_2\geq 0$. They developed method of upper-lower solutions in both reversed and well ordered case. Urus et al. \cite{NUAKVMS1} explored this technique for the above BVPs when $\psi$ is independent of $u'$.

In this paper we investigate existence of solution for the following BVPs,
\begin{eqnarray}
\label{ppr2sec1eq1}&&-u''(x)=\psi(x,u,u'), \quad 0<x<1,\\
\label{ppr2sec1eq2}&&u'(0)=\lambda_{1}u(\xi), \quad u'(1)=\lambda_{2} u(\eta),
\end{eqnarray}
where $I=[0,1]$, $0<\xi\leq\eta<1$, $\lambda_{1},\lambda_{1}\geq 0$  and $\psi(x,u,u')$ is continuous. We develop maximum and anti-maximum principle, MI-Technique with upper-lower solution. We prove existence and approximation of solution for both reversed and well ordered case and compute the region of existence. We obtain that the non-linear source term $\psi(x,u,u')$ is Lipschitz in $u'$ and one sided Lipschitz in $u$. To this end we verify all the conditions and examples by using Mathematica 11.3. This paper is divided in four sections. In the second section we explore corresponding linear BVPs. In the third section we study the existence of non-linear problem for the case $0<k<\pi^2/4$ and by numerical illustration we prove our results. Similarly in fourth section we study for the case where $k<0$.
\section{ Quasilinearization}\label{section1}
In this section we define an iterative scheme for the BVP \eqref{ppr2sec1eq1} and \eqref{ppr2sec1eq2} which is given as follows,
\begin{eqnarray*}
&& -u''_{n+1}(x)-k u_{n+1}(x)=\psi(x,u_{n}(x),u'_{n}(x))-k u_{n}(x), \\
&&u'_{n+1}(0)=\lambda_{1}u_{n+1}(\xi), \quad u'_{n+1}(1)=\lambda_{2} u_{n+1}(\eta),
\end{eqnarray*}
where $\psi\in C(I\times\mathbb{R}\times\mathbb{R},\mathbb{R})$, $0<\xi\leq\eta<1$, $\lambda_1,\lambda_2\geq 0$, $n\in \mathbb{N}$ and $k\in \mathbb{R}-\{0\}$ is constant.
The linear BVPs corresponding to above iterative scheme is:
\begin{eqnarray}
\label{ppr2sec2eq1} &&-u''(x)-k u(x)=g(x) ,\quad  0<x<1,\\
\label{ppr2sec2eq2} &&u'(0)=\lambda_{1}u(\xi), \quad u'(1)=\lambda_{2} u(\eta)+c,
\end{eqnarray}
where $g=\psi(x,u,u')-k u$ is continuous in $[0,1]$ and constant $c \geq 0$.\\
\section{For $k$ Positive, i.e., $0<k<\pi^{2}/4$.}
\label{ppr2sec3}
This section is divided into five sub sections. In first subsection we derive Green's function, sign of Green's function, solution of BVPs \eqref{ppr2sec2eq1}-\eqref{ppr2sec2eq2} and anti-maximum principle. In second subsection we prove existence of some differential inequality which is used to determine monotonicity of sequences of upper-lower solutions,  in subsection three MI-Technique has been developed in reverse ordered case, also some lemmas and proposition have been developed which we use to prove the existence. In fourth subsection we obtain bound for the derivative of solution. Then we establish our main theorem which proves the existence of solutions between upper and lower solutions. In the last subsection, we give examples and show that all the sufficient conditions are true for a given range of $k$ and monotonic sequences exists and converge to the solution of the nonlinear problem.
\subsection{Deduction of Green's Function}\label{ppr2sec3subsec1}
We consider linear  BVPs \eqref{ppr2sec2eq1}-\eqref{ppr2sec2eq2}
\begin{eqnarray}
\label{ppr2sec3eq1}
&&-u''(x)-k u(x)=g(x) ,\quad  x\in (0,1),\\
\label{ppr2sec3eq2}
&&u'(0)=\lambda_{1}u(\xi), \quad u'(1)=\lambda_{2} u(\eta),
\end{eqnarray}
where $g(x)$ is continuous in $[0,1]$.
\begin{lemma}
\label{ppr2sec3lemma1} If Green's function of the BVPs \eqref{ppr2sec3eq1}-\eqref{ppr2sec3eq2} is $G(x,s)$. Then $G(x,s)$ is given by
$$G(x,s)=\frac{1}{\sqrt{k}D_{k}}
\begin{cases}
\sqrt{k}\cos\sqrt{k}x(\sqrt{k}\cos\sqrt{k}(s-1)+\lambda_{2}\sin\sqrt{k}(s-\eta))\\+\lambda_{1}\sin\sqrt{k}(s-x)(\lambda_{2}\sin\sqrt{k}(\eta-\xi)-\sqrt{k}\cos\sqrt{k}(\xi-1)), & 0\leq x\leq s\leq\xi; \\
\sqrt{k}\{\cos\sqrt{k}s(\sqrt{k}\cos\sqrt{k}(x-1)+\lambda_{2}\sin\sqrt{k}(x-\eta))\}, & 0\leq x,s\leq\xi; \\
(\sqrt{k}\cos\sqrt{k}x+\lambda_1\sin\sqrt{k}(x-\xi))(\sqrt{k}\cos\sqrt{k}(s-1)+\lambda_2 \sin\sqrt{k}(s-\eta)), & \xi\leq x\leq s\leq\eta; \\
(\sqrt{k}\cos\sqrt{k}s+\lambda_1\sin\sqrt{k}(s-\xi))(\sqrt{k}\cos\sqrt{k}(x-1)+\lambda_2\sin\sqrt{k}(x-\eta)), & s\leq x, s\leq\eta; \\
\sqrt{k}\cos\sqrt{k}(s-1)(\sqrt{k}\cos\sqrt{k}x+\lambda_1\sin\sqrt{k}(x-\xi)), & \eta\leq x\leq s\leq 1; \\
\sqrt{k}\cos\sqrt{k}(x-1)(\sqrt{k}\cos\sqrt{k}s+\lambda_1\sin\sqrt{k}(s-\xi))\\+\lambda_2\sin\sqrt{k}(x-s)(\sqrt{k}\cos\sqrt{k}\eta+\lambda_1\sin\sqrt{k}(\eta-\xi)), & s\leq x, s\leq 1,
\end{cases}
$$
where $ D_k=k\sin\sqrt{k}+\lambda_{2}\sqrt{k}\cos\sqrt{k}\eta+\lambda_{1}\big(\lambda_{2}\sin\sqrt{k}(\eta-\xi)-\sqrt{k}\cos\sqrt{k}(\xi-1)\big).$
\end{lemma}
\begin{proof}

Proof is similar to the proof described in \cite{verma2019monotone}.
\end{proof}
Let us assume that\\

$ [A_1]:$  $ D_k>0,~ \sqrt{k}\cos\sqrt{k}-\lambda_2 \sin\sqrt{k}\eta\geq 0$, $\sqrt{k}-\lambda_1 \sin\sqrt{k}\xi>0.$\\

In section \ref{P2sec4}, we have shown graphically that above inequalities satisfy in $k \in (\alpha,\beta)\subseteq (0,\pi^2/4)$.
\begin{lemma}
\label{ppr2sec3lemma2} Suppose $[A_1]$ is satisfied, then $G(x,s)\geq0.$
\end{lemma}
\begin{proof}
Since $[A_1]$ holds $\forall~ \xi,\eta\in[0,1]$. Let us prove for the interval $0\leq x\leq s\leq \xi$. We have from Green's function,
\begin{eqnarray*}
a_{1}=&&\frac{1}{\sqrt{k}D_k}\{{\lambda_{2}\sqrt{k}\sin\sqrt{k}(s-\eta)+k\cos\sqrt{k}(s-1)+\lambda_{1}\sin\sqrt{k}s\{\lambda_{2}\sin\sqrt{k}(\eta-\xi)-\sqrt{k}\cos\sqrt{k}(\xi-1)\}\}},\\
\Rightarrow a_{1}=&&\frac{1}{\sqrt{k}D_k}[(\sqrt{k}\sin\sqrt{k}+\lambda_{2}\cos\sqrt{k}\eta)(\sqrt{k}-\lambda_{1}\sin\sqrt{k}\xi)\sin\sqrt{k}s+(\sqrt{k}\cos\sqrt{k}-\lambda_{2}\sin\sqrt{k}\eta)\\
&&(\sqrt{k}\cos\sqrt{k}s-\lambda_{1}\sin\sqrt{k}s\cos\sqrt{k}\xi)].
\end{eqnarray*}
As, $\cos\sqrt{k}\xi\leq\cos\sqrt{k}s$ and $\sin\sqrt{k}s\leq\sin\sqrt{k}\xi$. Now we have,
\begin{eqnarray*}
\cos\sqrt{k}\xi(\sqrt{k}-\lambda_{1}\sin\sqrt{k}\xi)\leq\cos\sqrt{k}\xi(\sqrt{k}-\lambda_{1}\sin\sqrt{k}s)\leq\sqrt{k}\cos\sqrt{k}s-\lambda_{1}\sin\sqrt{k}s\cos\sqrt{k}\xi.
\end{eqnarray*}
Hence,
\begin{eqnarray*}
 a_{1}\geq&&\frac{(\sqrt{k}-\lambda_{1}\sin\sqrt{k}\xi)}{\sqrt{k}D_k}[(\sqrt{k}\sin\sqrt{k}+\lambda_{2}\cos\sqrt{k}\eta)\sin\sqrt{k}s+(\sqrt{k}\cos\sqrt{k}-\lambda_{2}\sin\sqrt{k}\eta)\cos\sqrt{k}\xi].
\end{eqnarray*}
Now we have,
\begin{eqnarray*}
b_{1}&&= \frac{1}{\sqrt{k}D_k}\lambda_{1}\cos\sqrt{k}s[\sqrt{k}\cos\sqrt{k}(\xi-1)-\lambda_{2}\sin\sqrt{k}(\eta-\xi)],\\
&&=\frac{1}{\sqrt{k}D_k}\lambda_{1}\cos\sqrt{k}s[\cos\sqrt{k}\xi(\sqrt{k}\cos\sqrt{k}-\lambda_{2}\sin\sqrt{k}\eta)+\sin\sqrt{k}\xi(\sqrt{k}\sin\sqrt{k}+\lambda_{2}\cos\sqrt{k}\eta)].
\end{eqnarray*}
By applying $[A_1]$ it can be easily seen that  $a_{1},~ b_{1}\geq0$. Hence $G(x,s)\geq0$, for $0\leq x\leq s\leq\xi$. In similar fashion we can easily prove for other intervals.
\end{proof}
\begin{lemma}
\label{ppr2sec3lemma3} If $g(x)$ is continuous in $[0,1]$ and $c\geq 0$ is any constant, then the solution $u(x)\in C^2(0,1)$ of  BVP \eqref{ppr2sec2eq1} and \eqref{ppr2sec2eq2} is given by
\begin{eqnarray}
\label{ppr2sec3eq3}u(x)=\displaystyle{\frac{-c}{D_k}\big(\sqrt{k}\cos\sqrt{k}x+\lambda_1 \sin\sqrt{k}(x-\xi)\big)}-\int_{0}^{1}G(x,s)g(s)ds.
\end{eqnarray}

\end{lemma}
\begin{proof} It is easy to deduce by using the concept of CF (Complimentary Function) and PI (Particular Integral).
\end{proof}
\begin{prop}{\textbf{Anti-Maximum Principle:}}
\label{ppr2sec3prop1} If $g(x)\geq0$ is continuous in $[0,1]$ and $c\geq 0$ is any constant. Suppose $[A_1]$ is satisfied, then the solution $u(x)$ given in equation \eqref{ppr2sec3eq3} is non-positive.
\end{prop}
\begin{proof} Given that $g(x)\geq0$, $c\geq 0$ and $[A_1]$ is satisfied. Now \eqref{ppr2sec3eq3} can be written as,
\begin{eqnarray*}
u(x)=\displaystyle{\frac{-c}{D_k}\big(\cos\sqrt{k}x(1-\lambda_1 \sin\sqrt{k}\xi)+\lambda_1 \cos\sqrt{k}\xi \sin\sqrt{k}x\big)}-\int_{0}^{1}G(x,s)g(s)ds.
\end{eqnarray*}
 Applying  $[A_1]$ and lemma \ref{ppr2sec3lemma2} in above equation we can easily obtain the required result.
\end{proof}
\subsection{Existence of Some Differential Inequalities}\label{ppr2sec3subsec2}
\begin{lemma}
\label{ppr2sec3lemma4}
Suppose $L_1\in \mathbb{R}^+$ and $L_2(x):[0,1]\rightarrow \mathbb{R}^+ $ are such that
\begin{enumerate}[label=(\roman*)]
 \item $ L_1-k\leq 0,$
\item $L_2(0)=0,~ L_2'(x)\geq 0.$
\end{enumerate}
 then the following inequalities hold,
\begin{itemize}
\item[(a)] If $ (L_1-k)\cos\sqrt{k}+L_2(x)\sqrt{k}\sin\sqrt{k}\leq 0,$ then $$Y_1(x)= (L_1-k)\cos\sqrt{k}x+L_2(x)\sqrt{k}\sin\sqrt{k}x\leq 0, ~~ \forall x\in [0,1]. $$
\item[(b)] If $(L_1-k)+\text{sup} L_2'(x)\leq 0,$ then $$Y_2(x)= (L_1-k)\sin\sqrt{k}x+L_2(x)\sqrt{k}\cos\sqrt{k}x\leq 0, ~~ \forall x\in [0,1]. $$
\end{itemize}
\end{lemma}
\begin{proof} We divide the proof in two parts.\\
(a) We know that $\cos x$ is decreasing and $\sin x$ is increasing function in $(0,\frac{\pi}{2}).$ Using these properties we have, $$ (L_1-k)\cos\sqrt{k}x+L_2(x)\sqrt{k}\sin\sqrt{k}x\leq (L_1-k)\cos\sqrt{k}+L_2(x)\sqrt{k}\sin\sqrt{k},~ \forall x\in [0,1].$$
The desired result follows from the assumption.\\
(b) We have, $Y'_{2}(x)=\sqrt{k}\big((L_1-k)+L_2'(x)\big)\cos\sqrt{k}x-k L_2(x)\sin\sqrt{k}x\leq 0$, whenever $(L_1-k)+\text{sup} L_2'(x)\leq 0$. Therefore, $Y_2(x)$ is decreasing $\forall x\in [0,1]$  and $Y_2(0)=0$. Hence the desired result.
\end{proof}
\begin{lemma}
\label{ppr2sec3lemma4a} Suppose $[A_1]$ and conditions of lemma \ref{ppr2sec3lemma4} are satisfied then the following inequalities hold,
\begin{itemize}
\item[(a)] $ (L_1-k)(\sqrt{k}\cos\sqrt{k}x+\lambda_1 \sin\sqrt{k}(x-\xi))\pm L_2(x)\sqrt{k}(\sqrt{k}\sin\sqrt{k}x-\lambda_1 \cos\sqrt{k}(x-\xi))\leq 0,~~ \forall x\in [0,1]$.
\item[(b)] $(L_1-k) G(x,s)\pm L_2(x)\frac{\partial G(x,s)}{\partial x}\leq 0;\quad x\neq s,\quad \forall x\in [0,1]$.
\end{itemize}
\end{lemma}
\begin{proof}
(a) Consider the positive sign case, i.e.,
\begin{eqnarray*}
&& (L_1-k)(\sqrt{k}\cos\sqrt{k}x+\lambda_1 \sin\sqrt{k}(x-\xi))+ L_2(x)\sqrt{k}(\sqrt{k}\sin\sqrt{k}x-\lambda_1 \cos\sqrt{k}(x-\xi))\\
=&&(\sqrt{k}-\lambda_1 \sin\sqrt{k}\xi)((L_1-k)\cos\sqrt{k}x+L_2(x)\sqrt{k}\sin\sqrt{k}x)+\lambda_1 \cos\sqrt{k}\xi((L_1-k)\sin\sqrt{k}x-L_2(x)\sqrt{k}\cos\sqrt{k}x).
\end{eqnarray*}
By using inequality $(a)$ of lemma \ref{ppr2sec3lemma4} we conclude the result. Similarly we can also prove for negative sign case.\\

(b) Consider the positive sign case, i.e.,
 \begin{eqnarray}
\label{ppr2sec3eqd1}&&(L_1-k) G(x,s)+ L_2(x)\frac{\partial G(x,s)}{\partial x}, \quad x\neq s.
\end{eqnarray}
To evaluate sign of \eqref{ppr2sec3eqd1} we first evaluate $\frac{\partial G}{\partial x}$, $x \neq s$, from lemma \ref{ppr2sec3lemma1} for each interval individually. Then we substitute the values of $G(x,s)$ and $\frac{\partial G}{\partial x}$  for each sub interval of $[0,1]$ in equation \eqref{ppr2sec3eqd1}.\\
For brevity, let us define,
\begin{eqnarray*}
Z_{1}&=&\sqrt{k}\cos\sqrt{k}(s-1)+\lambda_{2}\sin\sqrt{k}(s-\eta),\\
Z_{2}&=&\sqrt{k}\cos\sqrt{k}(\xi-1)+\lambda_{2}\sin\sqrt{k}(\xi-\eta),\\
Z_{3}&=&\sqrt{k}\cos\sqrt{k}s+\lambda_1\sin\sqrt{k}(s-\xi),\\
Y_{3}&=&(L_1-k)\cos\sqrt{k}x-L_2(x)\sqrt{k}\sin\sqrt{k}x.
\end{eqnarray*}
By simple calculations it can be easily seen that $Z_{1},~Z_{2},~Z_{3}\geq 0$ and  $Y_{3}(x)\leq 0.$\\
(i) When $0\leq x\leq s\leq\xi:$
  Then $~G(x,s)$ and $\frac{\partial{G}}{\partial{x}}$ can be written as,
\begin{eqnarray*}
&&G(x,s)=Z_{1}\sqrt{k}\cos\sqrt{k}x-Z_{2}\lambda_{1}\sin\sqrt{k}(s-x),\\
&&\frac{\partial{G(x,s)}}{\partial{x}}=-Z_{1}k\sin\sqrt{k}x+Z_{2}\lambda_{1}\sqrt{k}\cos\sqrt{k}(s-x),\quad x\neq s.
\end{eqnarray*}
Now equation \eqref{ppr2sec3eqd1} becomes,
\begin{eqnarray*}
&&(L_1-k) G(x,s)+ L_2(x)\frac{\partial G}{\partial x}=Y_{1}(x)(Z_{1}\sqrt{k}-Z_{2}\lambda_{1}\sin\sqrt{k}s)+Y_{2}(x)Z_{2}\lambda_{1}\cos\sqrt{k}s,\quad x\neq s.
\end{eqnarray*}
where $ Y_{1}(x) $ and $ Y_{2}(x) $ are given in lemma  \ref{ppr2sec3lemma4}. Substituting values of $Z_{1}$ and $Z_{2}$ in $(Z_{1}\sqrt{k}-Z_{2}\lambda_{1}\sin\sqrt{k}s)$ and simplifying we get,
\begin{eqnarray*}
Z_{1}\sqrt{k}-Z_{2}\lambda_{1}\sin\sqrt{k}s&=&(\sqrt{k}\cos\sqrt{k}-\lambda_{2}\sin\sqrt{k}\eta)(\sqrt{k}\cos\sqrt{k}s-\lambda_{1}\sin\sqrt{k}s\cos\sqrt{k}\xi)\\
&\quad& +(\sqrt{k}\sin\sqrt{k}+\lambda_{2}\cos\sqrt{k}\eta)(\sqrt{k}\sin\sqrt{k}s-\lambda_{1}\sin\sqrt{k}s\sin\sqrt{k}\xi),\\
\Rightarrow Z_{1}\sqrt{k}-Z_{2}\lambda_{1}\sin\sqrt{k}s
&\geq& \cos\sqrt{k}\xi(\sqrt{k}\cos\sqrt{k}-\lambda_{2}\sin\sqrt{k}\eta)(\sqrt{k}-\lambda_{1}\sin\sqrt{k}s)\\
&\quad& +\sin\sqrt{k}s(\sqrt{k}\sin\sqrt{k}+\lambda_{2}\cos\sqrt{k}\eta)(\sqrt{k}-\lambda_{1}\sin\sqrt{k}\xi).
\end{eqnarray*}
Applying inequality $[A_1]$ and lemma \ref{ppr2sec3lemma4} we obtain that equation \eqref{ppr2sec3eqd1} is non positive.\\
(ii) When $0\leq s\leq x\leq\xi:~$ From lemma \ref{ppr2sec3lemma1}\\
\begin{eqnarray*}
&&\frac{\partial G(x,s)}{\partial x}=\sqrt{k}\cos\sqrt{k}s(-k\sin\sqrt{k}(x-1)+\lambda_{2}\sqrt{k}\cos\sqrt{k}(x-\eta)),\quad x\neq s.
\end{eqnarray*}
Now equation \eqref{ppr2sec3eqd1} becomes,
\begin{eqnarray*}
(L_1-k) G(x,s)+ L_2(x)\frac{\partial G}{\partial x}&=&\sqrt{k}\cos\sqrt{k}s\big(\sqrt{k}\{(L_1-k)\cos\sqrt{k}(x-1)-L_2(x)\sqrt{k}\sin\sqrt{k}(x-1)\}\\
&\quad& +\lambda_{2}\{(L_1-k)\sin\sqrt{k}(x-\eta)+L_2(x)\sqrt{k}\cos\sqrt{k}(x-\eta)\}\big),\\
&=&\sqrt{k}\cos\sqrt{k}s \big(Y_3(x)(\sqrt{k}\cos\sqrt{k}-\lambda_{2}\sin\sqrt{k}\eta)
+Y_{2}(x)(\sqrt{k}\sin\sqrt{k}+\lambda_{2}\cos\sqrt{k}\eta)\big).
\end{eqnarray*}
Applying inequality $[A_1]$ and (b) of lemma \ref{ppr2sec3lemma4} we get the required result.\\
(iii) When $\xi\leq x\leq s\leq\eta:$
\begin{eqnarray*}
&&G(x,s)=Z_{1}(\sqrt{k}\cos\sqrt{k}x+\lambda_{1}\sin\sqrt{k}(x-\xi)),\\
&&\frac{\partial G(x,s)}{\partial x}=Z_{1}(-{k}\sin\sqrt{k}x+\lambda_{1}\sqrt{k}\cos\sqrt{k}(x-\xi)),\quad x \neq s.
\end{eqnarray*}
Now for $x \neq s$,
\begin{eqnarray*}
(L_1-k) G(x,s)+ L_2(x)\frac{\partial G}{\partial x}&=&Z_{1}\big((L_1-k)(\sqrt{k}\cos\sqrt{k}x+\lambda_{1}\sin\sqrt{k}(x-\xi))\\
&\quad& +L_2(x)\sqrt{k}(-\sqrt{k}\sin\sqrt{k}x+\lambda_{1}\cos\sqrt{k}(x-\xi)) \big)\\
&=&Z_{1}\big(\sqrt{k} Y_{3}(x)+\lambda_{1}\big(Y_{2}(x)\cos\sqrt{k}\xi- Y_{3}(x)\sin\sqrt{k}\xi\big)\big).
\end{eqnarray*}
Applying inequality $[A_1]$ and (b) of lemma \ref{ppr2sec3lemma4} we obtain that equation \eqref{ppr2sec3eqd1} is non-positive for this case.\\
(iv) When $\xi\leq s\leq x \leq\eta:$
\begin{eqnarray*}
&&G(x,s)=Z_3\big(\sqrt{k}\cos\sqrt{k}(x-1)+\lambda_2\sin\sqrt{k}(x-\eta)\big),\\
&&\frac{\partial G(x,s)}{\partial x}=Z_3\big(-k\sin\sqrt{k}(x-1)+\sqrt{k}\lambda_2\cos\sqrt{k}(x-\eta)\big), \quad x \neq s.
\end{eqnarray*}
\begin{eqnarray*}
(L_1-k) G(x,s)+ L_2(x)\frac{\partial G}{\partial x}&=&Z_3[\sqrt{k}\{(L_1-k)\cos\sqrt{k}(x-1)-L_2(x)\sqrt{k}\sin\sqrt{k}(x-1)\}\\
&\quad& +\lambda_2 \{(L_1-k)\sin\sqrt{k}(x-\eta)+L_2(x)\sqrt{k}\cos\sqrt{k}(x-\eta)\}]\\
&=&Z_3[\sqrt{k}\big(Y_3(x) \cos\sqrt{k}+Y_2(x)\sin\sqrt{k}\big)+\lambda_2\big(Y_2(x)\cos\sqrt{k}\eta-Y_3(x) \sin\sqrt{k}\eta\big)]\\
&=&Z_3\big(Y_3(x)(\sqrt{k}\cos\sqrt{k}-\lambda_2\sin\sqrt{k}\eta)+Y_2(x)(\sqrt{k}\sin\sqrt{k}+\lambda_2\cos\sqrt{k}\eta)\big)\leq0.
\end{eqnarray*}
(v) When $\eta\leq x\leq s\leq 1:$
\begin{eqnarray*}
&&G(x,s)=\sqrt{k}\cos\sqrt{k}(s-1)(\sqrt{k}\cos\sqrt{k}x+\lambda_1\sin\sqrt{k}(x-\xi)),\\
&&\frac{\partial G(x,s)}{\partial x}=k\cos\sqrt{k}(s-1)\big(-\sqrt{k}\sin\sqrt{k}x+\lambda_1\cos\sqrt{k}(x-\xi)\big), \quad x \neq s.
\end{eqnarray*}
Equation \eqref{ppr2sec3eqd1} implies,
\begin{eqnarray*}
(L_1-k) G(x,s)+ L_2(x)\frac{\partial G}{\partial x}=\sqrt{k}\cos\sqrt{k}(s-1)[\sqrt{k}Y_3(x)+\lambda_1\big(Y_2(x)\cos\sqrt{k}\xi-Y_3(x)\sin\sqrt{k}\xi\big)].
\end{eqnarray*}
Clearly this equation is $\leq 0.$\\
(vi) When $\eta \leq s\leq x\leq 1:$
\begin{eqnarray*}
&&G(x,s)=Z_3\sqrt{k}\cos\sqrt{k}(x-1)+\lambda_2\sin\sqrt{k}(x-s)(\sqrt{k}\cos\sqrt{k}\eta+\lambda_1\sin\sqrt{k}(\eta-\xi)),\\
&&\frac{\partial G(x,s)}{\partial x}=-Z_3k\sin\sqrt{k}(x-1)+\sqrt{k}\lambda_2\cos\sqrt{k}(x-s)(\sqrt{k}\cos\sqrt{k}\eta+\lambda_1\sin\sqrt{k}(\eta-\xi)), \quad x \neq s.
\end{eqnarray*}
Equation \eqref{ppr2sec3eqd1} implies,
\begin{eqnarray*}
(L_1-k)G(x,s)+ L_2(x)\frac{\partial G}{\partial x}&=&Z_3\sqrt{k}\big((L_1-k)\cos\sqrt{k}(x-1)-L_2(x)\sqrt{k}\sin\sqrt{k}(x-1)\big)+\lambda_2 \big(\sqrt{k}\cos\sqrt{k}\eta,\\
&\quad& +\lambda_1\sin\sqrt{k}(\eta-\xi)\big)\big((L_1-k)\sin\sqrt{k}(x-s)+L_2(x)\sqrt{k}\cos\sqrt{k}(x-s)\big)\\
&=&Z_3\sqrt{k}(Y_2(x)\cos\sqrt{k}+Y_3(x)\sin\sqrt{k})+\lambda_2(\sqrt{k}\cos\sqrt{k}\eta+\lambda_1\sin\sqrt{k}(\eta-\xi))\\
&\quad& (Y_3(x)\cos\sqrt{k}-Y_2(x)\sin\sqrt{k}s).
\end{eqnarray*}
Applying inequality $[A_1]$ and lemma \ref{ppr2sec3lemma4}, it is easy to show the above equation is non positive. Hence proof is complete for positive sign case. Similarly we can prove for negative sign case.

\end{proof}
\subsection{Non-Well Ordered Case: Construction of Upper-Lower Solutions}\label{ppr2sec3subsec3}
In this section upper-lower solutions are defined and some conditions on $c(x),~ d(x)$ and $\psi(x,u,u')$ are assumed. Then we define the sequence of functions $\{c_n(x)\}_n$ and $\{d_n(x)\}_n$ and develop monotone iterative method based on these sequences. We prove some lemmas which show that upper solutions are monotonically increasing and lower solutions are monotonically decreasing. Also we develop a theorem which gives that the sequence of functions $\{c_n(x)\}_n$ and $\{d_n(x)\}_n$ are uniformly convergent and converge to the solution of BVPs \eqref{ppr2sec1eq1}-\eqref{ppr2sec1eq2} under some sufficient conditions.
\begin{defn}
\label{ppr2sec3de1} Lower solution of BVPs \eqref{ppr2sec2eq1}-\eqref{ppr2sec2eq2} is defined as,
 \begin{eqnarray*}
&&-c{''}(x)\leq \psi(x,c,c') , \quad 0<x<1,\\
&&c'(0)= \lambda_{1}c(\xi), \quad c'(1)\leq \lambda_{2} c(\eta),
\end{eqnarray*}
where $c(x)\in C^2[0,1].$
\end{defn}
\begin{defn} \label{ppr2sec3de2} Upper solution of BVP \eqref{ppr2sec2eq1}-\eqref{ppr2sec2eq2} is defined as,
\begin{eqnarray*}
&&-d''(x)\geq \psi(x,d,d') , \quad 0<x<1,\\
&&d'(0)= \lambda_{1}d(\xi), \quad d'(1)\geq \lambda_{2} d(\eta),
\end{eqnarray*}
where $d(x)\in C^2[0,1]$.
\end{defn}
Let us assume some conditions as follows,\\

$[A_2]:$ there exist lower solution  $c(x)$ and upper solution $d(x)$ of  BVP \eqref{ppr2sec2eq1}-\eqref{ppr2sec2eq2} where $c(x),~ d(x) \in C^2(I)$\\ such that
 $$ c(x)\geq d(x) ~~\forall x\in I;$$

$[A_3]:$ $\psi(x,v,w):E\rightarrow \mathbb{R}$ is continuous function on $E:=\{(x,v,w)\in I\times \mathbb{R}^2:d(x)\leq v\leq c(x)\};$\\

$[A_4]:$ $\exists$ $L_1\geq 0$ such that $\forall~ (x,v_1,w),~ (x,v_2,w) \in E,$
$$v_1\leq v_2\Rightarrow \psi(x,v_2,w)-\psi(x,v_1,w)\leq L_1(v_2-v_1);$$

$[A_5]:$ $\exists$  $L_2(x):[0,1]\rightarrow \mathbb{R}^+ $ such that $L_2(x)\geq 0,~ L_2'(x)\geq 0$ and $\forall ~(x,v,w_1),~ (x,v,w_2) \in E,$
$$\mid \psi(x,v,w_1)-\psi(x,v,w_2)\mid \leq L_2(x)\mid(w_1-w_2)\mid.$$

Also we propose sequence of functions $\{c_n(x)\}_n$ and $\{d_n(x)\}_n$ such that,
$$c_0(x)=c(x),$$
\begin{eqnarray}
\label{ppr2sec3eq4}&&-c''_{n+1}(x)-k c_{n+1}= \psi(x,c_n,c'_n)-k c_{n},\\
\label{ppr2sec3eq5}&&c'_{n+1}(0)= \lambda_{1}c_{n+1}(\xi), \quad c'_{n+1}(1)= \lambda_{2} c_{n+1}(\eta),
\end{eqnarray}
and $$d_0(x)=d(x),$$
\begin{eqnarray}
\label{ppr2sec3eq6}&&-d''_{n+1}(x)-k d_{n+1}= \psi(x,d_n,d'_n)-k d_{n},\\
\label{ppr2sec3eq7}&&d'_{n+1}(0)=\lambda_{1}d_{n+1}(\xi), \quad d'_{n+1}(1)= \lambda_{2} d_{n+1}(\eta).
\end{eqnarray}

\begin{lemma}
\label{ppr2sec3lemma6} If $c_n(x)$ is lower solution of \eqref{ppr2sec2eq1}-\eqref{ppr2sec2eq2}. Then $c_n(x)\geq c_{n+1}(x),~~\forall x\in I,$ where $c_{n+1}(x)$ is defined by equation \eqref{ppr2sec3eq4}-\eqref{ppr2sec3eq5}.
\end{lemma}
\begin{proof}
Given that $c_n(x)$ is lower solution of \eqref{ppr2sec2eq1}-\eqref{ppr2sec2eq2} and $c_{n+1}(x)$ is given by equation \eqref{ppr2sec3eq4}-\eqref{ppr2sec3eq5}. So by using these assumptions we have,
\begin{eqnarray*}
&&-(c''_{n+1}-c''_n)-k(c_{n+1}-c_n)\geq 0, \quad n\in \mathbb{N},\\
&&(c_{n+1}-c_n)'(0)= \lambda_{1}(c_{n+1}-c_n)(\xi), \quad (c_{n+1}-c_n)'(1)\geq \lambda_{2} (c_{n+1}-c_n)(\eta).
\end{eqnarray*}
This is in the form of equations \eqref{ppr2sec2eq1}-\eqref{ppr2sec2eq2} with solution $u(x),$
where $u(x)= c_{n+1}-c_n ,~  g(x)=-(c''_{n+1}-c''_n)-k(c_{n+1}-c_n)\geq 0$ and $c\geq 0$. Hence result can be concluded from proposition \ref{ppr2sec3prop1} easily.
\end{proof}

\begin{prop}
\label{ppr2sec3prop2} Suppose $L_1\in \mathbb{R}^+$, $L_2(x):[0,1]\rightarrow \mathbb{R}^+ $ are such that $[A_1]-[A_5]$ and conditions of lemma \ref{ppr2sec3lemma4} hold then $\forall~ x\in [0,1]$, if $c_n(x)$ is lower solution of \eqref{ppr2sec2eq1}-\eqref{ppr2sec2eq2} then,
 $$(k-L_1)u(x)+L_2(sign~(u')) u'\leq 0,$$
 where $u(x)$ is any solution of BVPs \eqref{ppr2sec2eq1}-\eqref{ppr2sec2eq2}.
\end{prop}
\begin{proof}
Let $u(x)= c_{n+1}-c_n.$ Given that  $c_n(x)$ satisfies \eqref{ppr2sec2eq1}-\eqref{ppr2sec2eq2} and $c_{n+1}$ is given by equation \eqref{ppr2sec3eq4}-\eqref{ppr2sec3eq5}. So by using these equations we have,
\begin{eqnarray*}
&&-u''(x)-ku(x)=-c''_{n+1}(x)+c''_n(x)-k c_{n+1}(x)+k c_n(x)=c''_n(x)+\psi(x,c_n,c'_n)\geq 0,\\
\text{and,}&& (c_{n+1}-c_n)'(0)= \lambda_{1}(c_{n+1}-c_n)(\xi), \quad (c_{n+1}-c_n)'(1)\geq \lambda_{2} (c_{n+1}-c_n)(\eta).
\end{eqnarray*}
 This is in the form of equations \eqref{ppr2sec2eq1}-\eqref{ppr2sec2eq2} with solution $u(x)$ therefore $u(x)$ can be written in the form of \eqref{ppr2sec3eq3} with $g(x)=c''_n(x)+\psi(x,c_n,c'_n)$. Hence to prove given inequality we substitute the value of $u(x)$ and $u'(x)$ from \eqref{ppr2sec3eq3} we obtain that
 \begin{multline*}
(k-L_1)u(x)+L_2(sign~(u')) u'= \frac{b}{D_k}[(L_1-k)(\sqrt{k}\cos\sqrt{k}x+\lambda_1 \sin\sqrt{k}(x-\xi))\pm L_2(x)\sqrt{k}(\sqrt{k}\sin\sqrt{k}x\\
 -\lambda_1 \cos\sqrt{k}(x-\xi))]+\int_{0}^{1}\Big[(L_1-k)G(x,s)\pm L_2(x)\frac{\partial G(x,s)}{\partial x}\Big]g(s)ds.
 \end{multline*}
 By inequality (a), (b) of lemma \ref{ppr2sec3lemma4a} result can be concluded easily.
\end{proof}
\begin{lemma}
\label{ppr2sec3lemma7}Suppose $L_1\in \mathbb{R}^+$, $L_2(x):[0,1]\rightarrow \mathbb{R}^+ $ are such that $[A_1]-[A_5]$ and conditions of lemma \ref{ppr2sec3lemma4} hold then the function $c_n(x)$ given by equation \eqref{ppr2sec3eq4}-\eqref{ppr2sec3eq5} satisfy,
\begin{itemize}
\item[(a)] $c_n(x)\geq c_{n+1}(x)$,
\item[(b)] $c_n(x)$ is lower solution of \eqref{ppr2sec2eq1}-\eqref{ppr2sec2eq2}.
\end{itemize}
\end{lemma}
\begin{proof}
By using recurrence we prove monotonicity of $c_n(x)$.\\
Step 1: If $n=0$, $ c_0(x)=c(x) $, where $c(x)$ is lower solution of \eqref{ppr2sec2eq1}-\eqref{ppr2sec2eq2}, therefore by lemma \ref{ppr2sec3lemma6} we have $c_1\leq c_0.$\\
Step 2: Suppose for $ n-1$, $c_{n-1}(x)$ is lower solution of \eqref{ppr2sec2eq1}-\eqref{ppr2sec2eq2} and $c_{n}\leq c_{n-1}.$\\
By definition of lower solutions,
 \begin{eqnarray*}
&&c_{n-1}''(x)+\psi(x,c_{n-1},c_{n-1}')\geq 0 , \quad x\in [0,1],\\
&&c_{n-1}'(0)= \lambda_{1}c_{n-1}(\xi), \quad c_{n-1}'(1)\leq \lambda_{2} c_{n-1}(\eta).
\end{eqnarray*}
To show that $c_n$ is lower solution of \eqref{ppr2sec2eq1}-\eqref{ppr2sec2eq2} we have,
\begin{eqnarray*}
-c''_{n}-\psi(x,c_{n},c'_{n})&=&-\psi(x,c_{n},c'_{n})+\psi(x,c_{n},c'_{n})+k c_n-k c_{n-1},\\
&\leq& L_1(c_{n-1}-c_n)+L_2(x)\mid c'_{n-1}-c'_n \mid +k(c_n-c_{n-1}),\\
&\leq&(k-L_1)(c_n-c_{n-1})+L_2(x) \mid c'_{n-1}-c'_n \mid.
\end{eqnarray*}
Let $u= c_{n}-c_{n-1} $. By using proposition  \ref{ppr2sec3prop2} we arrive at,
\begin{eqnarray*}
-c''_{n}-\psi(x,c_{n},c'_{n})\leq0.
\end{eqnarray*}
This proves that $c_n$ is lower solution of \eqref{ppr2sec2eq1}-\eqref{ppr2sec2eq2} and  therefore by  lemma \ref{ppr2sec3lemma6} $c_{n+1}\leq c_{n}.$
\end{proof}
\begin{lemma}
\label{ppr2sec3lemma8} If $d_n(x)$ is upper solution of \eqref{ppr2sec2eq1}-\eqref{ppr2sec2eq2}. Then $d_n(x)\leq c_{n+1}(x)$, $\forall x\in [0,1]$, where $d_{n+1}(x)$ is defined by equation \eqref{ppr2sec3eq6}-\eqref{ppr2sec3eq7}.
\end{lemma}
\begin{proof}
Proof is similar to lemma \ref{ppr2sec3lemma6}.
\end{proof}
\begin{lemma}
\label{ppr2sec3lemma9}Suppose $L_1\in \mathbb{R}^+$, $L_2(x):[0,1]\rightarrow \mathbb{R}^+ $ are such that $[A_1]-[A_5]$ and conditions of lemma \ref{ppr2sec3lemma4} hold then the function $d_n(x)$ given by equation \eqref{ppr2sec3eq6}-\eqref{ppr2sec3eq7} satisfy,
\begin{itemize}
\item[(a)] $d_n(x)\leq d_{n+1}(x)$,
\item[(b)] $d_n(x)$ is upper solution of \eqref{ppr2sec2eq1}-\eqref{ppr2sec2eq2}.
\end{itemize}
\end{lemma}
\begin{proof}
Proof is similar to lemma \ref{ppr2sec3lemma7}.
\end{proof}
\begin{prop}
\label{ppr2subsec3prop3}Suppose $L_1\in \mathbb{R}^+$, $L_2(x):[0,1]\rightarrow \mathbb{R}^+ $ and $L_2(0)=0$ are such that $[A_1]-[A_5]$ and conditions of lemma \ref{ppr2sec3lemma4} hold and $$\psi(x,d(x),d'(x))-\psi(x,c(x),c'(x))-k(d(x)-c(x))\geq 0,$$
then $c_n\geq d_n$, $\forall x\in [0,1],$
where $c_n$ and $d_n$ are given by equation \eqref{ppr2sec3eq4}-\eqref{ppr2sec3eq5} and \eqref{ppr2sec3eq6}-\eqref{ppr2sec3eq7} respectively.
\end{prop}
\begin{proof}
Define, $ g_i (x) = \psi(x, d _i, b ' _i) - \psi(x, c _i, c ' _i)- k (d _i - c _i),~\forall i \in \mathbb{N}$\\
and let $ u_i = d _i - c _i$ which satisfies,
\begin{eqnarray*}
&&-u''_i - k u_i = \psi(x, d_{i-1}, d'_{i-1}) - \psi(x, c _{i-1}, c' _{i-1}) - k ( d_{i-1} -c _{i-1} )= g_{i-1} (x),\\
&& ( d _i -c _i )'(0)=\lambda_1(d _i -c _i)(\xi), \quad ( d _i -c_i )'(1)=\lambda_2(d _i -c_i)(\eta).
\end{eqnarray*}
\emph{Claim 1.} $c _1 \geq d _1.$ For $i = 1$, we have
\begin{eqnarray*}
&&-u_1'' - k u_1 = g_0(x) \geq 0\\
&&(d_1 - c_1)'(0) =\lambda_1(d _1 -c _1)(\xi), \quad ( d _1 -c _1 )'(1)\geq\lambda_2(d _1 -c _1)(\eta).
\end{eqnarray*}
Therefore $u_1$ is a solution of eqs \eqref{ppr2sec2eq1}-\eqref{ppr2sec2eq2}  with $g(x) = g_0(x)$. Hence by Proposition \ref{ppr2sec3prop1},
$c_1 \geq d_1 $.\\

\noindent \emph{Claim 2.} Suppose $g_{n-2} \geq 0$ and $c_{n-1} \geq d_{n-1}$.\\
 Now ,
\begin{eqnarray*}
g_{n-1} (x)&=& \psi(x, d _{n-1}, d ' _{n-1}) - \psi(x, c _{n-1}, c ' _{n-1})- k (c _{n-1} - c _{n-1})\\
&\geq& -(k - L_1)(d _{n-1} - c _{n-1})-L_2(x)|c' _{n-1}-d'_{n-1}|\\
&\geq& -[(k - L_1) u_{n-1} + L_2(x)(\text{sign} u'_{n-1}) u'_{n-1}].
\end{eqnarray*}
With the help of Proposition \ref{ppr2sec3prop2} we can prove that,
\begin{eqnarray*}
&& (k - L_1) u_{n-1} + L_2(x)(sign~u'_{n-1}) u'_{n-1}\leq0.
\end{eqnarray*}
Therefore $g_{n-1}\geq0$, also we have $u_n = d_{n} - c_{n}$ for $i=n$. Then $u_n$ satisfy,
\begin{eqnarray*}
&&-u''_n - k u_n =g_{n-1}(x)\geq0.\\
&&(d_{n} - c_{n})'(0) = \lambda_1(d _n -c _n)(\xi), \quad ( d _n -c _n )'(1)\geq\lambda_2(d _n -c _n)(\eta).
\end{eqnarray*}
We deduce from Proposition \ref{ppr2sec3prop1} that $d_{n} \leq c_{n}$.
\end{proof}
\subsection{Bound on Derivative of Solution} \label{ppr2sec3.4}
$[A_6]:$ \textbf{Nagumo Condition:} Let $ \mid \psi(x,v,w)\mid\leq\phi(\mid w \mid);~ \forall (x,v,w) \in E,$ where  $ \phi:\mathbb{R}^{+}\rightarrow \mathbb{R}^{+} $ is continuous which satisfiy,
$$\int_{\gamma}^{\infty}\frac{sds}{\phi(s)} \geq \max_{x\in [0,1]} c(x) - \min_{x\in [0,1]} d(x),$$
such that $ \gamma=2\displaystyle\sup_{x\in [0,1]}\mid c(x)\mid$.
\begin{lemma}
\label{ppr2sec3lemma10} Let $ [A_6] $  be true then there exist $ P>0 $ such that $ \parallel u'\parallel_\infty\leq P$, $\forall x\in [0,1],$
where $ u$ is any solution of inequality
\begin{eqnarray}
\label{ppr2sec3eq8}&&-u''(x)\geq \psi(x,u,u') ,\quad x\in I,\\
\label{ppr2sec3eq9}&&u'(0)= \lambda_1 u(\xi),\quad u'(1)\geq \lambda_1 u(\eta),
\end{eqnarray}
such that $ d(x)\leq u(x)\leq c(x).$
\end{lemma}
\begin{proof}
Proof is divided into three cases.

\item[Case-1:] If $u(x)$ is monotonically increasing in $(0,1)$:
By Mean value theorem there exist $\alpha \in (0,1)$ such that
\begin{eqnarray*}
&& u'(\alpha)= u(1)-u(0),\\
\Rightarrow && |u'(\alpha)|\leq 2\sup_{x\in [0,1]} |u(x)|\leq 2\sup_{x\in [0,1]} |c(x)|\leq \gamma,~ \text{where}~ \gamma= 2 \sup_{x\in [0,1]} |c(x)|.
\end{eqnarray*}
Using $ \mid \psi(x,v,w)\mid\leq\phi(\mid w \mid)$ in equation \eqref{ppr2sec3eq8} and integrating from limit $\alpha$ to $x$ the equation becomes,
\begin{eqnarray*}
&&\int_{\alpha}^{x}\frac{u''(x)u'(x)}{\phi (|u'(x)|)}dx\leq \int_{\alpha}^{x}u'(x) dx\leq \max_{x\in I} c(x)- \min_{x\in I
} d(x).\\
\end{eqnarray*}
Let $u'(x)=s(>0),$  since $|u'(\alpha)|\leq \gamma$ then,
\begin{eqnarray*}
&&\int_{\gamma}^{u'(x)}\frac{s ds}{\phi (s)}\leq \int_{u'(\alpha)}^{u'(x)}\frac{s ds}{\phi (s)}\leq \max_{x\in I} c(x)- \min_{x\in I
} d(x).\\
\end{eqnarray*}
Adding $\int_{0}^{\gamma}\frac{s ds}{\phi (s)}$ in both side of above equation and applying condition $ [A_6] $ we obtain,  $$ \parallel u'\parallel_\infty\leq P ~~\forall x\in [0,1].$$
\item[Case-2:] If $u(x)$ is monotonically decreasing in $(0,1)$: Proof of this case is similar to case-1.
\item[Case-3:] If $u(x)$ is neither monotonically decreasing nor monotonically increasing in $[0,1]$:
Proof of this case is divided into two sub-cases.
\item[Subcase-1:] Consider the interval $(x_0,x]\subset (0,1)$ such that $u'(x_0)=0$ and $u'(x)>0$ for $x>x_0.$
Using $ \mid \psi(x,v,w)\mid\leq\phi(\mid w \mid)$ in equation \eqref{ppr2sec3eq8} and integrating from $x_0$ to $x$ the equation becomes,
\begin{eqnarray*}
&&\int_{x_0}^{x}\frac{u''(x)u'(x)}{\phi (|u'(x)|)}dx\leq \int_{x_0}^{x}u'(x) dx\leq \max_{x\in I} c(x)- \min_{x\in I
} d(x).\\
\end{eqnarray*}
Let $u'(x)=s(>0),$ choose $P>0$ and using condition $ [A_6] $ we obtain,
\begin{eqnarray*}
&&\int_{0}^{u'(x)}\frac{sds}{\phi (s)}\leq \max_{x\in I} c(x)- \min_{x\in I
} d(x)\leq \int_{\gamma}^{P}\frac{sds}{\phi (s)}\le\int_{0}^{P}\frac{sds}{\phi (s)}\\
\end{eqnarray*}
$$\Rightarrow   \parallel u'\parallel_\infty\leq P ~~\forall x\in [0,1].$$
\item[Subcase-2:] Consider the interval $[x,x_0)\subset (0,1)$ such that  $u'(x)<0$ for $x<x_0.$ \\
\indent Proof is similar to Subcase-1.
\end{proof}
\begin{lemma}
\label{ppr2sec3lemma11}  Let $ [A_6] $  be true then there exist $ P>0 $ such that $ \parallel u'\parallel_\infty\leq P$, $\forall x\in [0,1]$, where $ u$ is any solution of inequality
\begin{eqnarray}
\label{ppr2sec3eq10}&&-u''(x)\leq \psi(x,u,u') ,\quad x\in I,\\
\label{ppr2sec3eq11}&&u'(0)= \lambda_1 u(\xi),\quad u'(1)\leq \lambda_1 u(\eta),
\end{eqnarray}
such that $ d(x)\leq u(x)\leq c(x).$
\end{lemma}
\begin{proof}
Proof is similar to lemma \ref{ppr2sec3lemma10}.
\end{proof}
\begin{theorem}
\label{ppr2sec3subsec3the1} Suppose $L_1\in \mathbb{R}^+$, $L_2(x):[0,1]\rightarrow \mathbb{R}^+ $ and $L_2(0)=0$ are such that $[A_1]-[A_5]$ and conditions of lemma \ref{ppr2sec3lemma4} hold and $\forall x\in [0,1]$
$$\psi(x,d(x),d'(x)-\psi(x,c(x),c'(x))-k(d-c)\geq0,$$
then $ (c_n)_n \rightarrow y$ and $ (d_n)_n\rightarrow z$  uniformly in $C^1[0,1]$ such that
 $d\leq y\leq z\leq c,$
where $y(x)$ and $z(x)$ are solutions of \eqref{ppr2sec2eq1} and \eqref{ppr2sec2eq2}.
\end{theorem}
\begin{proof}
We have already proved that the sequences $(c_n)_n $ and $ (d_n)_n $ are such that,
\begin{eqnarray}
\label{ppr2sec3eq12}c=c_0\geq c_1...\geq c_n...\geq d_n\geq...\geq d_1\geq d_0 = d.
\end{eqnarray}
 Now we prove that the sequences $(c_n)_n $ and $ (d_n)_n $ converges uniformly in $ C^{1}[0,1] $ to solutions $ y $ and $ z $ of non linear BVP \eqref{ppr2sec2eq1}-\eqref{ppr2sec2eq1} such that $\forall x\in I $
 $$ d\leq y\leq z\leq c. $$
 Firstly, we prove that $(c_n)_n $ and $ (d_n)_n $ converges  in $ C^{1}([0,1]). $

 Since $(c_n)_n $ and $ (d_n)_n $ are bounded as well as monotonic therefore by monotone convergence theorem $(c_n)_n $ and $ (d_n)_n $ are convergent point wise. Let $ \lim_{n\rightarrow \infty} c_n(x) = y(x) $ and $ \lim_{n\rightarrow \infty}{ d_n(x)} = z(x).$ From equation \eqref{ppr2sec3eq12} and Lemma \ref{ppr2sec3lemma11} it can be deduced that $(c_n)_n $ is uniformly bounded and equicontinuous in  $ C^{2}[0,1].$
 i.e $\forall~ \epsilon>0$  $\exists$ $\delta>0$ such that $ \forall~  n$
 $$ |(c_n)(x)-(c_n)(y)| <\epsilon, \text{if}\quad  |x-y|< \delta.$$
  Therefore every sub-sequence $(c_{n_i})_i$ of $(c_n)_n $ is equibounded and equicontinuous in  $ C^{2}[0,1].$ We know  by Arzela-Ascoli theorem that there exist a sub-subsequence $(c_{n_{i_j}})_j$ of sub-sequence $(c_{n_i})_i$ which converges in $ C^{2}[0,1] .$
Since convergent sequences have unique limit point,
hence  $c_n(x)\rightarrow y(x) $ uniformly in $ C^{2}[0,1].$ Similarly, it can also be shown that $(d_n)_n(x)\rightarrow z(x) $ uniformly in $ C^{2}[0,1].$ Also in similar way by using Lemma \ref{ppr2sec3lemma10}-\ref{ppr2sec3lemma11}, equations \eqref{ppr2sec3eq4}-\eqref{ppr2sec3eq7} and Arzela-Ascoli theorem, we can prove that $c'_n(x)$ and $d'_n(x)$ are uniformly convergent and converges to $y'(x)$ and $z'(x)$ respectively.

We finally prove that $ y(x) $ and  $ z(x) $ are solutions of \eqref{ppr2sec2eq1}-\eqref{ppr2sec2eq1}. Since equations \eqref{ppr2sec3eq4}-\eqref{ppr2sec3eq5} and \eqref{ppr2sec3eq6}-\eqref{ppr2sec3eq7} are in the form of equations \eqref{ppr2sec3eq1}-\eqref{ppr2sec3eq2}, so the solution of these equations can be expressed as the form of equation \eqref{ppr2sec3eq3} for $(c_n)_n $ and $(d_n)_n $. After taking limit $n\rightarrow\infty$, and using Lebesgue dominated convergence theorem we can easily conclude that $ y(x) $ and  $ z(x) $ are the solutions of nonlinear BVP \eqref{ppr2sec1eq1}-\eqref{ppr2sec1eq2}. Hence the proof.
\end{proof}
 \section{Numerical Illustration} \label{P2sec4}
In this section for reverse order case we have considered an example. This example gives uniformly convergent sequences of upper and lower solutions which converges to the solution of our nonlinear problems for the specific range of $\frac{\partial \psi}{\partial u}  $.
\subsection{Example}
Consider four point BVPs,
\begin{eqnarray}
\label{p2sec4eq1}&&-u''(x)=\dfrac{e^u-x e^{u'}}{195},\\
\label{p2sec4eq2}&&u'(0)=2u(0.1),\quad u'(1)=3u(0.2),
\end{eqnarray}
where $ \psi(x,u,u')= \dfrac{e^u-x e^{u'}}{195}$, $ \xi=0.1,~ \eta=0.2,~\lambda_1=2 $ and $ \lambda_2=3. $ We consider initial lower and upper solutions as $ c(x)=1+2.525x+x^2,~d(x)=-(1+2.525x+x^2)$ respectively, where $ d(x) \leq c(x)$. Since $ \psi(x,u,u') $ is one sided Lipschitz in $ u $ so lipschitz constant is $ L_1=0.47331$ which is obtained by using $[A_4]$. Also $\psi$ is Lipschitz in $ u' $ therefore we derive from $[A_5]$, $ L_2(x)=\dfrac{x e^P}{195}$, where $ P>0 $ such that  $ \parallel u'\parallel_\infty\leq P, ~~\forall x\in [0,1]$. We obtain $ k\geq 0.4811 $ and from $[A_6]$, $ \phi(|s|)=\dfrac{e^{4.525}+e^{\mid s \mid}}{195} $,  $ P=0.2154 $. The range for $ k $ is computed by using these above results and Mathematica-11.3. From figure 1-4 it can be observed that there exist $(\alpha, \beta)\subset (0,\frac{\pi^2}{4})$ in which all the inequalities are satisfied and for this range the sequences are convergent which are shown in figure 5.\\

\textbf{Remark 1:} In figure 6, 7 we assume $k=1$, $2.3$ for which inequality shown in figure 3 is not valid but we are getting monotonic sequences.
\begin{figure}[H]
   \begin{minipage}{0.48\textwidth}
     \centering
     \includegraphics[width=.7\linewidth]{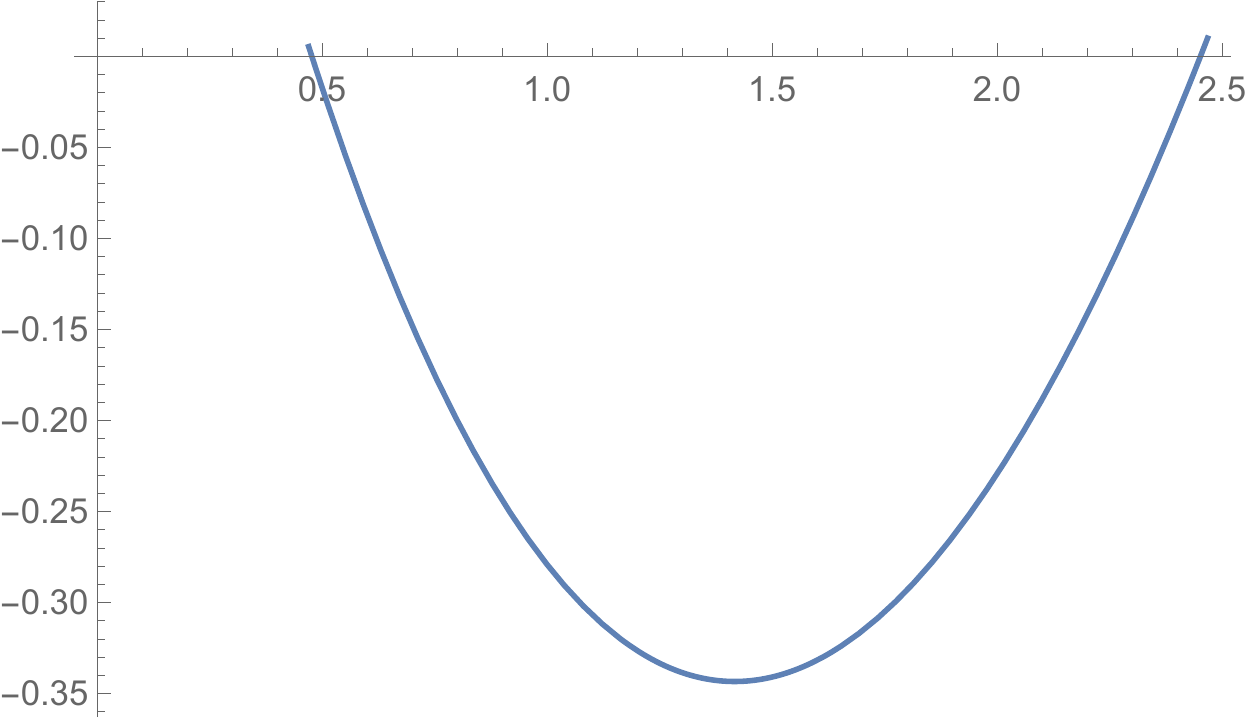}
     \caption{$(L_1-k)\cos\sqrt{k}+L_2(x)\sqrt{k}\sin\sqrt{k}$}\label{P2_Non_well_fig1}
   \end{minipage}\hfill
   \begin{minipage}{0.48\textwidth}
     \centering
     \includegraphics[width=.7\linewidth]{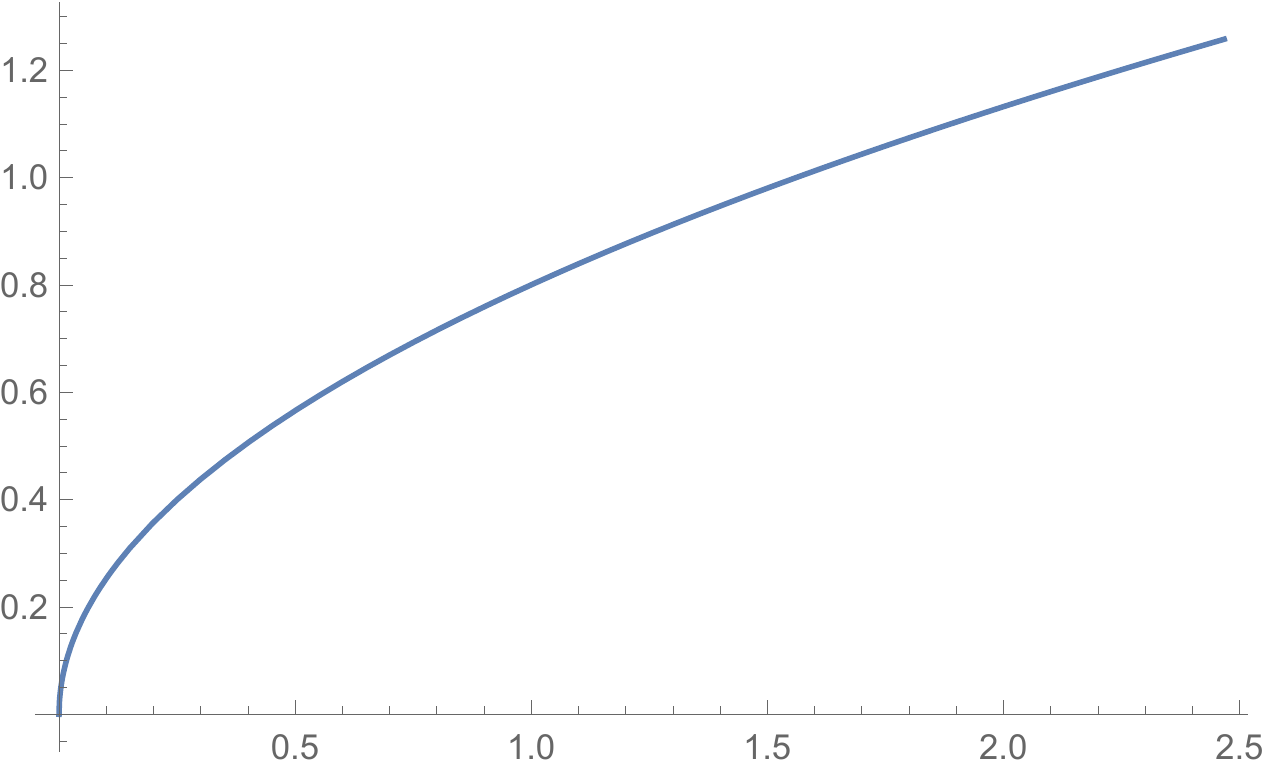}
     \caption{$\sqrt{k}-\lambda_1 \sin\sqrt{k}\xi$}\label{P2_Non_well_fig2}
   \end{minipage}
\end{figure}

\begin{figure}[H]
   \begin{minipage}{0.48\textwidth}
     \centering
     \includegraphics[width=.7\linewidth]{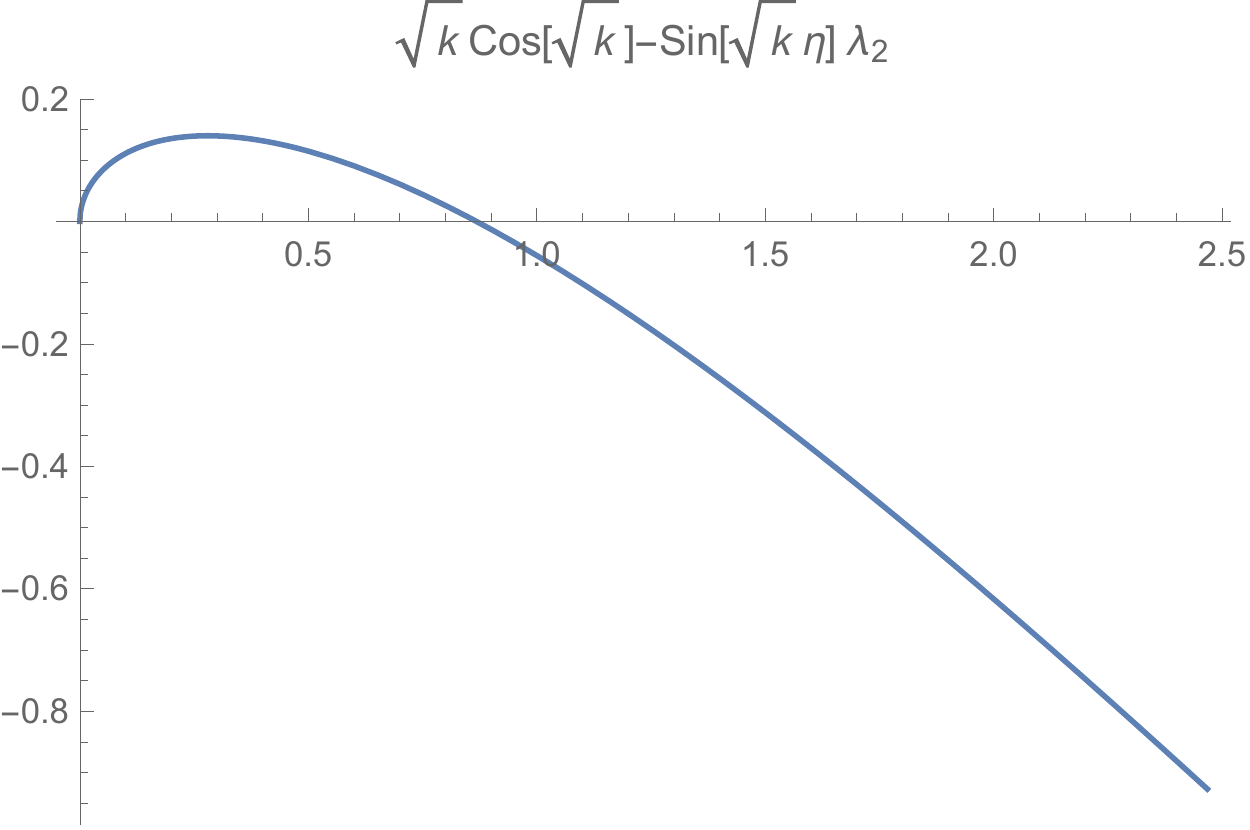}
     \caption{$\sqrt{k}\cos\sqrt{k}-\lambda_2 \sin\sqrt{k}\eta$}\label{P2_Non_well_fig3}
   \end{minipage}\hfill
   \begin{minipage}{0.48\textwidth}
     \centering
     \includegraphics[width=.7\linewidth]{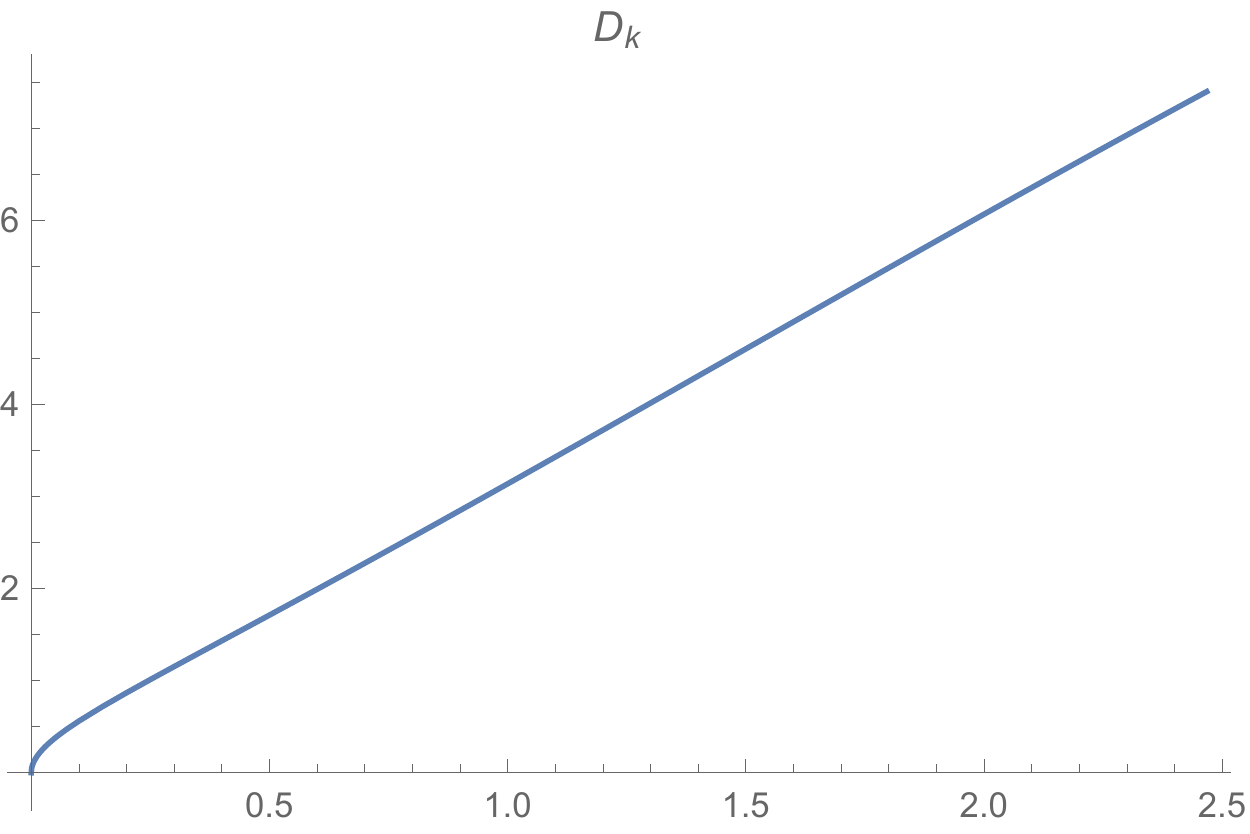}
     \caption{$D_k$}\label{P2_Non_well_fig4}
   \end{minipage}
\end{figure}

\begin{figure}[H]
   \begin{minipage}{0.48\textwidth}
     \centering
     \includegraphics[width=.7\linewidth]{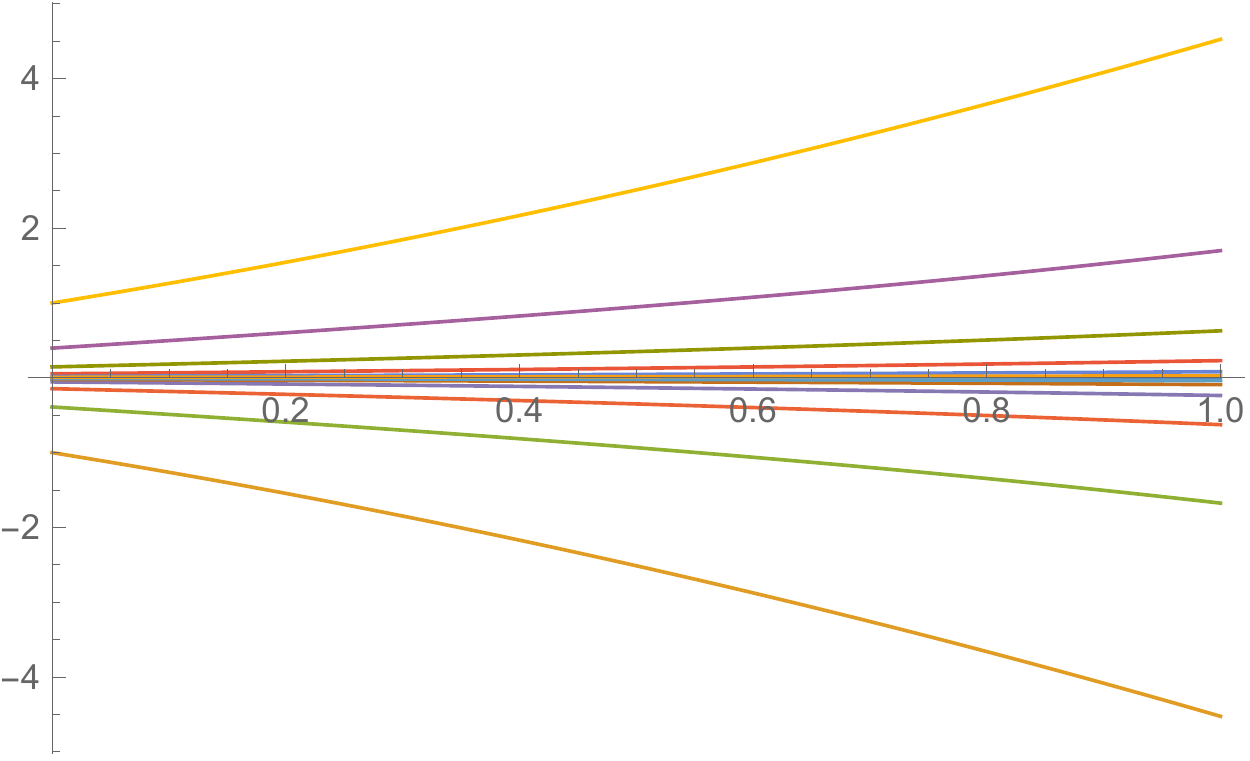}
     \caption{$k=0.49$}\label{P2_Non_well_seq1}
   \end{minipage}\hfill
   \begin{minipage}{0.48\textwidth}
     \centering
     \includegraphics[width=.7\linewidth]{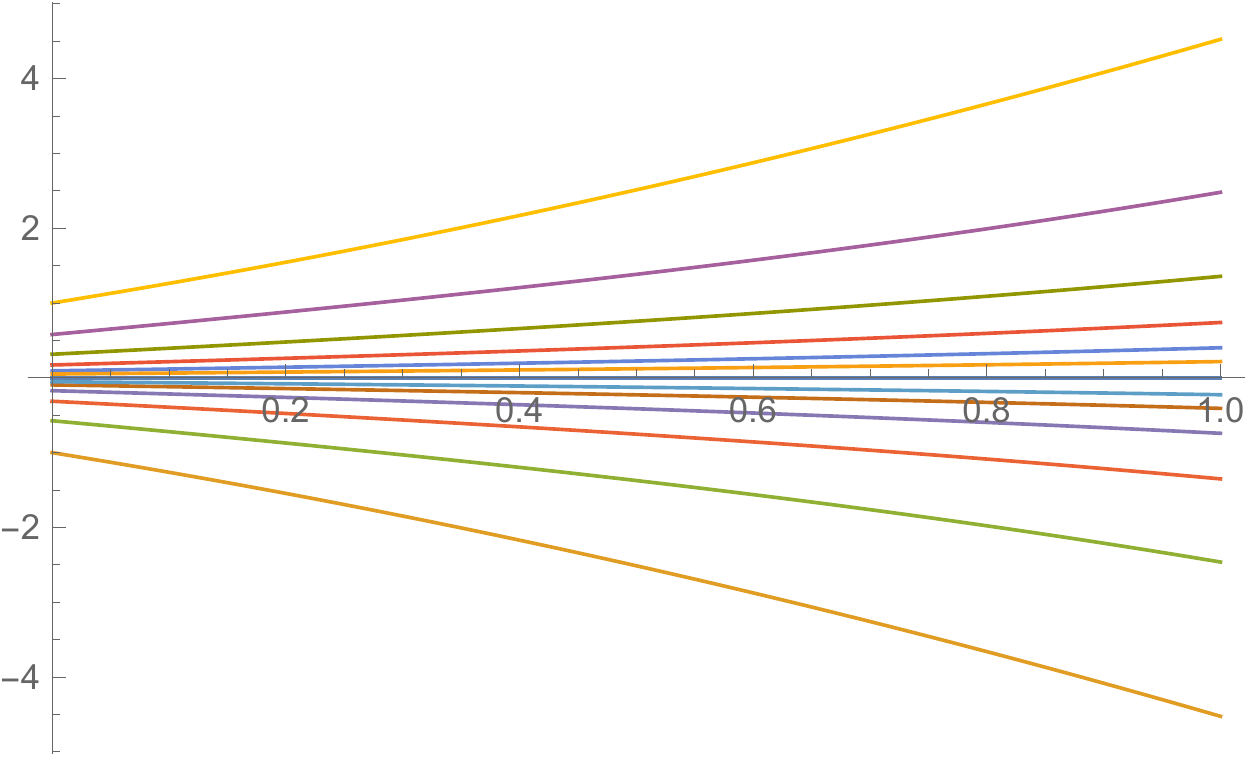}
     \caption{$k=1$}\label{P2_Non_well_seq2}
   \end{minipage}
\end{figure}

\begin{figure}[H]
\begin{center}
\includegraphics[width=8cm,height=6cm]{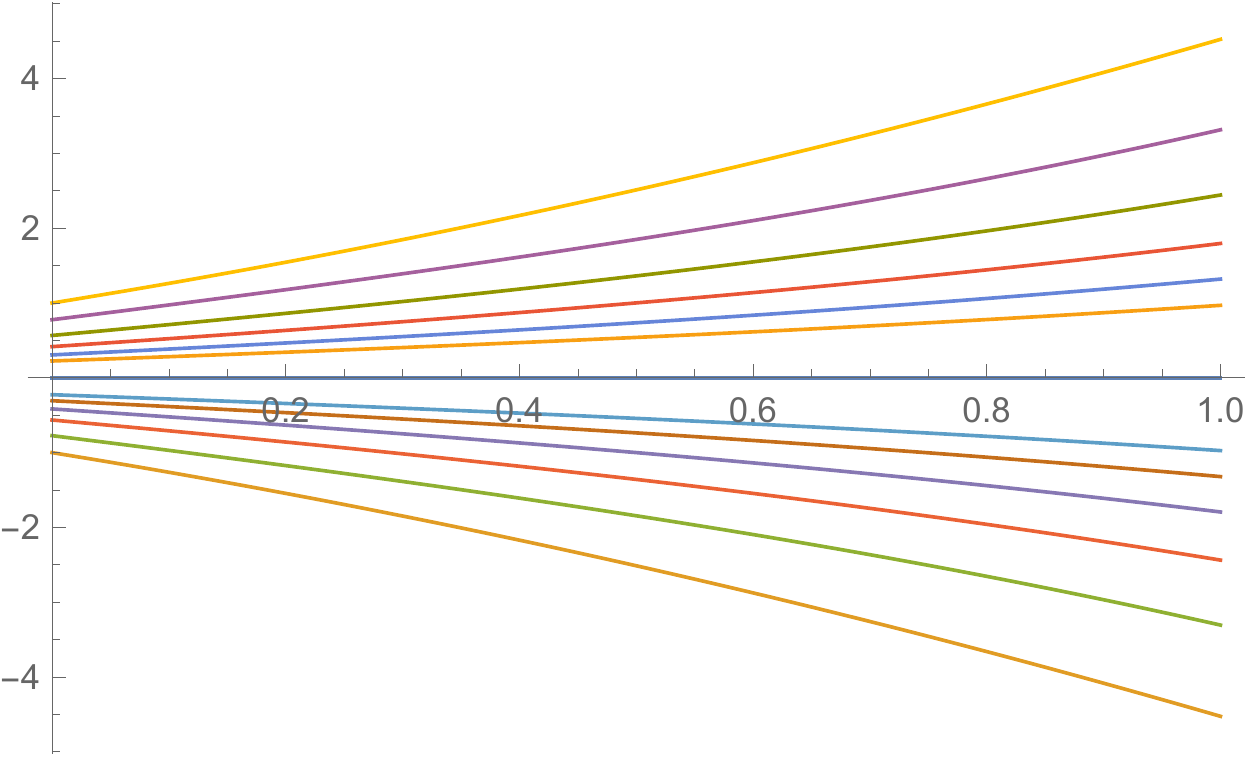}\caption{$k=2.3$}
\label{ppr1fig2}
\end{center}
\end{figure}
 \section{For $k$ Negative, i.e., $k<0$.} \label{ppr2sec5}

In this section for negative $k$ existence of BVPs \eqref{ppr2sec2eq1}-\eqref{ppr2sec2eq2} have been studied. This section is also similar to section three. It is also divided into five sub sections. In first subsection we derive Green's Function, sign of Green's function, solution of BVPs \eqref{ppr2sec2eq1}-\eqref{ppr2sec2eq2} and Maximum principle. In second Subsection we prove existence of some differential inequality which is used to prove monotonic behavior of sequences of well ordered upper-lower solutions. In subsection three MI-Technique has been developed also some lemmas and proposition have been given which we use to prove existence. In fourth subsection we show that the derivative of solution is bounded. Then we establish a theorem which proves convergence of solution between upper and lower solution. In the last subsection we give examples and compute the range of $k<0$ for which all the sufficient conditions are true and region of existence is also computed.
\subsection{Deduction of Green's Function}\label{ppr2section3}
\begin{lemma}
\label{ppr2sec4lemma1}If Green's function of the BVPs \eqref{ppr2sec3eq1}-\eqref{ppr2sec3eq2} is $G(x,s)$. Then $G(x,s)$ is given by
$$G(x,s)=\frac{1}{\sqrt{\mid k\mid}D_{k'}}
\begin{cases}
\sqrt{\mid k\mid}\cosh\sqrt{\mid k\mid}x(\lambda_{2}\sinh\sqrt{\mid k\mid}(\eta-s)-\sqrt{\mid k\mid}\cosh\sqrt{\mid k\mid}(s-1))\\+\lambda_{1}\sinh\sqrt{\mid k\mid}(s-x)(\sqrt{\mid k\mid}\cosh\sqrt{\mid k\mid}(\xi-1)-\lambda_{2}\sinh\sqrt{\mid k\mid}(\eta-\xi)), & 0\leq x\leq s\leq\xi; \\
-\sqrt{\mid k\mid}\{\cosh\sqrt{\mid k\mid}s(\sqrt{\mid k\mid}\cosh\sqrt{\mid k\mid}(x-1)+\lambda_{2}\sinh\sqrt{\mid k\mid}(x-\eta))\}, & s\leq x,s\leq\xi; \\
-(\sqrt{\mid k\mid}\cosh\sqrt{\mid k\mid}x+\lambda_1\sinh\sqrt{\mid k\mid}(x-\xi))(\sqrt{\mid k\mid}\cosh\sqrt{\mid k\mid}(s-1)
\\+\lambda_2 \sinh\sqrt{\mid k\mid}(s-\eta)), & \xi\leq x\leq s\leq\eta; \\
-(\sqrt{\mid k\mid}\cosh\sqrt{\mid k\mid}s+\lambda_1\sinh\sqrt{\mid k\mid}(s-\xi))(\sqrt{\mid k\mid}\cosh\sqrt{\mid k\mid}(x-1)\\
+\lambda_2\sinh\sqrt{\mid k\mid}(x-\eta)), & s\leq x, s\leq\eta; \\
-\sqrt{\mid k\mid}\cosh\sqrt{\mid k\mid}(s-1)(\sqrt{\mid k\mid}\cosh\sqrt{\mid k\mid}x+\lambda_1\sinh\sqrt{\mid k\mid}(x-\xi)), & \eta\leq x\leq s\leq 1; \\
\sqrt{\mid k\mid}\cosh\sqrt{\mid k\mid}(x-1)(\lambda_1\sinh\sqrt{\mid k\mid}(\xi-s)-\sqrt{\mid k\mid}\cosh\sqrt{\mid k\mid}s)\\+\lambda_2\sinh\sqrt{\mid k\mid}(s-x)(\sqrt{\mid k\mid}\cosh\sqrt{\mid k\mid}\eta+\lambda_1\sinh\sqrt{\mid k\mid}(\eta-\xi)), & s\leq x, s\leq 1,
\end{cases}
$$
where $D_{k'}=\mid k\mid\sinh\sqrt{\mid k\mid}-\lambda_{2}\sqrt{\mid k\mid}\cosh\sqrt{\mid k\mid}\eta-\lambda_{1}\lambda_{2}\sinh\sqrt{\mid k\mid}(\eta-\xi)+\lambda_{1}\sqrt{\mid k\mid}\cosh\sqrt{\mid k\mid}(\xi-1).$
\end{lemma}
\begin{proof}
Proof is similar to the proof described for $k<0$ in \cite{verma2019monotone}.
\end{proof}
Let us assume that \\

$[A'_1]:$  $\sqrt{\mid k\mid}\sinh\sqrt{\mid k\mid}-\lambda_2 \cosh\sqrt{\mid k\mid}\eta\geq0$, $\sqrt{\mid k\mid}\sinh\sqrt{\mid k\mid}\xi+(\lambda_1 -\sqrt{\mid k\mid})\cosh\sqrt{\mid k\mid}\xi\leq0$  and \\ \indent $\sqrt{\mid k\mid}-\lambda_1 \cosh\sqrt{\mid k\mid}\xi>0.$\\

In section \ref{section-6}, we have shown graphically that above inequalities can be satisfied when $k <0$.\\
\textbf{Remark 1:} $\sqrt{\mid k\mid}\sinh\sqrt{\mid k\mid}\xi+(\lambda_1 -\sqrt{\mid k\mid})\cosh\sqrt{\mid k\mid}\xi\leq0$ only if $ \lambda_1-\sqrt{|k|}\leq0.$

\begin{lemma}
\label{ppr2sec4lemma2} Suppose $[A'_1]$ is satisfied, then $G(x,s)\leq0$, $\forall x\in [0,1]$.
\end{lemma}
\begin{proof}
We first prove $D_k'>0,$  for this we have,
\begin{eqnarray*}
D_{k'}&&=\mid k\mid\sinh\sqrt{\mid k\mid}-\lambda_{2}\sqrt{\mid k\mid}\cosh\sqrt{\mid k\mid}\eta-\lambda_{1}\lambda_{2}\sinh\sqrt{\mid k\mid}(\eta-\xi)+\lambda_{1}\sqrt{\mid k\mid}\cosh\sqrt{\mid k\mid}(\xi-1)\\
&&=(\sqrt{\mid k\mid}\sinh\sqrt{\mid k\mid}-\lambda_2 \cosh\sqrt{\mid k\mid}\eta)(\sqrt{\mid k\mid}-\lambda_1 \sinh\sqrt{\mid k\mid}\xi)\\
&& \quad +\lambda_1 \cosh \sqrt{\mid k\mid}\xi(\sqrt{\mid k\mid}\cosh\sqrt{\mid k\mid}-\lambda_2 \sinh\sqrt{\mid k\mid}\eta)>0.
\end{eqnarray*}
 To prove  $G(x,s)\leq0$, $\forall x\in [0,1]$ we simplify $G(x,s)$ given in lemma \ref{ppr2sec4lemma1} for each sub interval of interval $[0,1]$ individually and using assumptions $[A'_1]$ we obtain $G(x,s)\leq0$.
\end{proof}
\begin{lemma}
\label{ppr2sec4lemma3} If $g(x)$ is continuous in $[0,1]$ and $c\geq 0$ is any constant, then the solution $u(x)\in C^2(0,1)$ of  BVP \eqref{ppr2sec2eq1} and \eqref{ppr2sec2eq2} is given by
\begin{eqnarray}
\label{ppr2sec4eq3}u(x)=\displaystyle{\frac{c}{D'_k}\big(\sqrt{\abs{k}}\cosh\sqrt{\abs{k}}x+\lambda_1 \sinh\sqrt{\abs{k}}(x-\xi)\big)}-\int_{0}^{1}G(x,s)g(s)ds.
\end{eqnarray}
\end{lemma}
\begin{proof}Proof is similar to the proof described in lemma 3.4 of \cite{verma2019monotone}.
\end{proof}
\begin{lemma}
\label{ppr2sec4lemma4}If $G(x,s)$ is Green's function of BVPs \eqref{ppr2sec2eq1}-\eqref{ppr2sec2eq2} then $\dfrac{\partial G}{\partial x}\leq 0$, $x \neq s.$

\end{lemma}
\begin{proof}
Since $G(x,s)$ satisfies the equation,
\begin{eqnarray}
\label{ppr2sec4eq*}&&u''(x)+k u(x)=0 ,\quad  0<x<1,\\
\label{ppr2sec4eq**}&&u'(0)=\lambda_{1}u(\xi), \quad u'(1)=\lambda_{2} u(\eta).
\end{eqnarray}
Integrating equation \eqref{ppr2sec4eq*} from $0$ to $x$ we have,
\begin{eqnarray*}
&& \int_0^{x} G''(x,s)dx=\int_0^{x} -k~ G(x,s)dx\\
\Rightarrow &&\dfrac{\partial G(x,s)}{•\partial x}=\lambda_1 G(\xi,s)-k  \int_0^{x} G(x,s)dx\leq0, \quad x \neq s,
\end{eqnarray*}
as $ k<0 $ and $G(x,s)\leq 0~~ \forall x \in [0,1].$

\end{proof}
\begin{prop}{\textbf{Maximum Principle:}}
\label{ppr2sec4prop1} If $g(x)\geq0$ is continuous in $[0,1]$ and $c\geq 0$ is any constant and $[A'_1]$ is satisfied then the solution $u(x)$ given in equation \eqref{ppr2sec4eq3} is non-negative.
\end{prop}
\begin{proof}
 Given that $g(x)\geq0$, $c\geq 0$ and $[A'_1]$  is satisfied. Now \eqref{ppr2sec4eq3} can be written as,
\begin{eqnarray*}
u(x)=\cosh\sqrt{\abs{k}}\xi\big(\abs{k}-\lambda_1 \sinh\sqrt{|k|}\xi\big)+\lambda_1 \sinh\sqrt{\abs{k}}x\cosh\sqrt{\abs{k}}x-\int_{0}^{1}G(x,s)g(s)ds.
\end{eqnarray*}
 Applying  $[A'_1]$ and lemma \ref{ppr2sec4lemma2} in above equation we can easily obtain the required result.
\end{proof}
\subsection{Existence of Some Differential Inequalities}\label{ppr2sec4subsec2}
\begin{lemma}
\label{ppr2sec4lemma5} Suppose $L_1\in \mathbb{R}^+,$  $ k<0$ are  such that $ L_1+k\leq 0$ and $L_2(x):[0,1]\rightarrow \mathbb{R}^+ $ is such that  $L_2(0)=0$, then the following inequalities hold,
\begin{itemize}
\item[(a)] If $(L_1+k)+\sup(L_2'(x)+L_2(x)\sqrt{\mid k \mid})\leq 0,$ then $$F(x)=(L_1+k)\sinh\sqrt{\mid k \mid}x+L_2(x)\sqrt{\mid k \mid}\cosh\sqrt{\mid k \mid}x\leq 0, ~~ \forall x\in [0,1]. $$
\item[(b)] $(L_1+k)\cosh\sqrt{\mid k \mid}x+L_2(x)\sqrt{\mid k \mid}\sinh\sqrt{\mid k \mid}x\leq 0, ~~ \forall x\in [0,1]. $
\end{itemize}
\end{lemma}
\begin{proof}
\begin{itemize}
\item[(a)] Since $ F(0)=0,$ and $ F'(x)\leq0$ whenever $(L_1+k)+\sup(L_2'(x)+L_2(x)\sqrt{\mid k \mid})\leq 0,~\forall x\in [0,1]$ and  therefore $ F(x)\leq0, ~\forall x\in [0,1].$ This completes the proof.

\item[(b)] Clearly, $(L_1+k)\cosh\sqrt{\mid k \mid}x+L_2(x)\sqrt{\mid k \mid}\sinh\sqrt{\mid k \mid}x\leq F(x)$. The result is obvious.
\end{itemize}
\end{proof}
\textbf{Remark 2:} From lemma \ref{ppr2sec4lemma5} (a), the inequality  $$(L_1+k)+\sup(L_2'(x)+L_2(x)\sqrt{\mid k \mid})\leq 0,$$ gives a bound for  $k$ such that,
   $$k\leq -\sup\bigg(L_1+L_2'(x)+\frac{{L_2(x)}^2}{2}+\frac{L_2(x)}{2}\sqrt{L_2^2(x)+4(L_1+L_2'(x))}\bigg),~\forall x\in [0,1]. $$
\begin{lemma}
\label{ppr2sec4lemma5a} Suppose $[A'_1]$ and conditions of lemma \ref{ppr2sec4lemma5} are satisfied, then $ \forall x\in [0,1] $ following inequalities hold,
\begin{itemize}
\item[(a)] $ (L_1+k)(\sqrt{\mid k \mid}\cosh\sqrt{\mid k \mid}x+\lambda_1 \sinh\sqrt{\mid k \mid}(x-\xi))\pm L_2(x)\sqrt{\mid k \mid}(\sqrt{k}\sinh\sqrt{\mid k \mid}x-\lambda_1 \cosh\sqrt{\mid k \mid}(x-\xi))\leq 0;$
\item[(b)] $(L_1+k) G(x,s)\pm L_2(x)\frac{\partial G(x,s)}{\partial x}\geq 0;\quad x\neq s,$ suppose to be the case for $(L_1+k)+\sup~ L_2(x)(\lambda_1-k)\leq 0.$
\end{itemize}
\end{lemma}
\begin{proof}
\begin{itemize}
\item[(a)]  Using lemma \ref{ppr2sec4lemma5} and inequality $ [A'_1] $, it is easy to prove the result.

\item[(b)] To prove $(L_1+k) G(x,s)+ L_2(x)\frac{\partial G(x,s)}{\partial x}\geq 0,$ we proceed similar to lemma \ref{ppr2sec3lemma4a} given in section \ref{ppr2sec3subsec2}.
\end{itemize}
\end{proof}
\textbf{Remark 3:} From condition lemma \ref{ppr2sec4lemma5a} (b) it can be observed that $$k\leq \frac{L_1+\lambda_1 \sup~ L_2(x)}{1-\sup L_2(x)},$$ only if $(1-\sup~ L_2(x))>0.$\\

From Remark 1, 2 and 3 we can conclude that,
 $$[A'_2]: k\leq \min \bigg\{-L_1,-\lambda_1^2,\frac{L_1+\lambda_1 \sup~ L_2(x)}{1-\sup~ L_2(x)},-\sup\bigg(L_1+L_2'(x)+\frac{{L_2(x)}^2}{2}+\frac{L_2(x)}{2}\sqrt{L_2^2(x)+4(L_1+L_2'(x))}\bigg)\bigg\}. $$

\subsection{Well Ordered Case: Construction of Upper-Lower  Solutions}
In this section we provide some assumptions based on lower-upper solutions and non linear term $\psi(x,u,u')$. We develop MI-Technique based on functions $\{c_n(x)\}_n$ and $\{d_n(x)\}_n$. We discuss some lemmas and propositions to shows that upper solutions are monotonically decreasing and lower solutions are monotonically increasing. We develop a theorem which gives that these sequence of  are uniformly converges to the solution of BVPs \eqref{ppr2sec1eq1}-\eqref{ppr2sec1eq2} under some sufficient conditions.\\

Assume the following properties,\\

$[A'_3]:$ there exist lower-upper solutions  $c(x)$ and $d(x)$ of  BVP \eqref{ppr2sec2eq1}-\eqref{ppr2sec2eq2} where $c(x)$, $d(x)$ $\in C^2[0,1]$ such that, $$ c(x)\leq d(x) ~~\forall x\in [0,1];$$

$[A'_4]:$  $\psi(x,v,w):E\rightarrow \mathbb{R}$ is continuous function on $E:=\{(x,v,w)\in I\times \mathbb{R}^2:c(x)\leq v\leq d(x)\};$\\

$[A'_5]:$ $\forall (x,v_1,w), (x,v_2,w) \in E,$ $\exists$  a constant $L_1\geq 0$ such that
$$v_1\leq v_2\Rightarrow \psi(x,v_2,w)-\psi(x,v_1,w)\geq -L_1(v_2-v_1);$$

$[A'_6]:$  $\forall (x,v,w_1), (x,v,w_2) \in E,$ $\exists$ a function $L_2(x)\geq 0$ such that
$$\mid \psi(x,v,w_1)-\psi(x,v,w_2)\mid \leq L_2(x)\mid(w_1-w_2)\mid.$$
\begin{lemma}
\label{ppr2sec4lemma7} If $c_n(x)$ is lower solution of \eqref{ppr2sec2eq1}-\eqref{ppr2sec2eq2}, then $c_n(x)\leq c_{n+1}(x)$, $\forall x\in [0,1]$, where $c_{n+1}(x)$ is defined by equation \eqref{ppr2sec3eq4}-\eqref{ppr2sec3eq5}.
\end{lemma}
\begin{proof}
To proof is similar to lemma \ref{ppr2sec3lemma6}.
\end{proof}

\begin{prop}
\label{ppr2sec4prop2} Suppose $L_1\in \mathbb{R}^+$, $L_2(x):[0,1]\rightarrow \mathbb{R}^+ $ and $L_2(0)=0$ are such that $[A'_1]-[A'_6]$ holds then $\forall x\in [0,1]$, if $c_n(x)$ is lower solution of \eqref{ppr2sec2eq1}-\eqref{ppr2sec2eq2} then,
 $$(L_1+k)u(x)+L_2(x)(sign~ u') u'\leq 0,$$
 where $u$ is any solution of BVPs \eqref{ppr2sec2eq1}-\eqref{ppr2sec2eq2}.
\end{prop}
\begin{proof}
Proof is similar to proposition \ref{ppr2sec3prop2} in section 3.
\end{proof}
\begin{lemma}
\label{ppr2sec4lemma8} Suppose $L_1\in \mathbb{R}^+$, $L_2(x):[0,1]\rightarrow \mathbb{R}^+ $ and $L_2(0)=0$ are such that $[A'_1]-[A'_6]$ and conditions of lemma \ref{ppr2sec4lemma5} hold then the function $c_n(x)$ given by equation \eqref{ppr2sec3eq4}-\eqref{ppr2sec3eq5} satisfy,
\begin{itemize}
\item[(a)] $c_n(x)\leq c_{n+1}(x)$,
\item[(b)] $c_n(x)$ is lower solution of \eqref{ppr2sec2eq1}-\eqref{ppr2sec2eq2}.
\end{itemize}
\end{lemma}
\begin{proof}
Proof is similar to lemma \ref{ppr2sec3lemma7} of section \ref{ppr2sec3}.
\end{proof}
\begin{lemma}
\label{ppr2sec4lemma9} If $d_n(x)$ is upper solution of \eqref{ppr2sec2eq1}-\eqref{ppr2sec2eq2}, then $d_n(x)\geq d_{n+1}(x)$, $\forall x\in [0,1]$, where $d_{n+1}(x)$ is defined by equation \eqref{ppr2sec3eq6}-\eqref{ppr2sec3eq7}.
\end{lemma}
\begin{lemma}
\label{ppr2sec4lemma10} Suppose $L_1\in \mathbb{R}^+$, $L_2(x):[0,1]\rightarrow \mathbb{R}^+ $ and $L_2(0)=0$ are such that $[A'_1]-[A'_6]$ and conditions of lemma \ref{ppr2sec4lemma5} hold then  the function  $d_n(x)$ given by equation \eqref{ppr2sec3eq6}-\eqref{ppr2sec3eq7} satisfy
\begin{itemize}
\item[(a)] $d_n(x)\geq d_{n+1}(x)$,
\item[(b)] $d_n(x)$ is upper solution of \eqref{ppr2sec2eq1}-\eqref{ppr2sec2eq2}.
\end{itemize}
\end{lemma}
\begin{proof}
Proof is similar to lemma \ref{ppr2sec3lemma8} of section \ref{ppr2sec3}.
\end{proof}
\begin{prop}
\label{ppr2sec4prop3} Suppose $L_1\in \mathbb{R}^+$, $L_2(x):[0,1]\rightarrow \mathbb{R}^+ $ and $L_2(0)=0$ are such that $[A'_1]-[A'_6]$ and conditions of lemma \ref{ppr2sec4lemma5} hold and
 $$\psi(x,d(x),d'(x))-\psi(x,c(x),c'(x))-k(d(x)-c(x))\geq 0,$$
then $c_n\leq d_n$, $\forall x\in [0,1]$,
where $c_n$ and $d_n$ are given by equation \eqref{ppr2sec3eq4}-\eqref{ppr2sec3eq5} and \eqref{ppr2sec3eq6}-\eqref{ppr2sec3eq7} respectively.
\end{prop}
\begin{proof}
Proof is same as proposition \ref{ppr2subsec3prop3} of section \ref{ppr2sec3}.
\end{proof}
\subsection{Bound on Derivative of Solution}\label{P2sec5subsec4}
$[A'_7]:$ \textbf{Nagumo condition:} Let $ \mid f(x,v,w)\mid\leq\phi(\mid w \mid);~ \forall (x,v,w) \in E,$ where  $ \phi:\mathbb{R}^{+}\rightarrow \mathbb{R}^{+} $ is continuous which satisfiy,
$$\int_{\gamma}^{\infty}\frac{sds}{\phi(s)} \geq \max_{x\in [0,1]} d(x) - \min_{x\in [0,1]} c(x),$$
such that $ \gamma=2\displaystyle\sup_{x\in [0,1]}\mid d(x)\mid$.

\begin{lemma}
\label{P2sec5subsec4lemma1} Let $ [A'_7] $  be true then there exist $ P>0 $ such that $ \parallel u'\parallel_\infty\leq P$ $\forall x\in [0,1]$, where $ u$ is any solution of inequality
\begin{eqnarray}
\label{ppr2sec4eq4}&&-u''(x)\geq \psi(x,u,u') ,\quad x \in I,\\
\label{ppr2sec4eq5}&&u'(0)= \lambda_1 u(\xi),\quad u'(1)\geq \lambda_1 u(\eta),
\end{eqnarray}
such that $ c(x) \leq u(x)\leq d(x).$
\end{lemma}
\begin{lemma}
\label{P2sec5subsec4lemma2}  Let $ [A'_7] $  be true then there exist $ P>0 $ such that $ \parallel u'\parallel_\infty\leq P$  $\forall x\in [0,1]$, where $ u$ is any solution of inequality
\begin{eqnarray}
\label{ppr2sec4eq6}&&-u''(x)\leq \psi(x,u,u') ,\quad x \in I,\\
\label{ppr2sec4eq7}&&u'(0)= \lambda_1 u(\xi),\quad u'(1)\leq \lambda_1 u(\eta),
\end{eqnarray}
such that $c(x)\leq u(x)\leq d(x).$
\end{lemma}
The proof of above two lemmas are similar to the proof described in lemma \ref{ppr2sec3lemma10} of section \ref{ppr2sec3.4}.
\begin{theorem}
\label{ppr2sec4the02} Suppose $L_1\in \mathbb{R}^+$, $L_2(x):[0,1]\rightarrow \mathbb{R}^+ $ and $L_2(0)=0$ are such that $[A'_1]-[A'_6]$ and conditions of lemma \ref{ppr2sec4lemma5} hold and $ \forall x\in[0,1]$
$$\psi(x,d(x),d'(x)-\psi(x,c(x),c'(x))-k(d-c)\geq0,~~~ \forall x\in [0,1],$$
then $ (c_n)_n\rightarrow y $ and $ (d_n)_n\rightarrow z $  uniformly in $C^1[0,1]$ such that
 $c\leq z\leq y\leq d$, where  $y$ and $z$ are solutions of \eqref{ppr2sec2eq1}-\eqref{ppr2sec2eq2}.
\end{theorem}
\section{Numerical Illustration}\label{section-6}
In this section we show numerically and graphically that for well order case sequences of upper and lower solutions are uniformly convergent  and converges to the solution. We also give some range for $k<0$ which will validate our results.
\subsection{Example}
Consider four point BVPs,
\begin{eqnarray}
\label{p2sec6eq1}&&-u''(x)=\dfrac{e^x-1}{40}\big(u'^2-u-\frac{\cos x}{4}\big),\\
\label{p2sec6eq2}&&u'(0)=\dfrac{1}{4}u(0.2),\quad u'(1)=\dfrac{1}{9}u(0.3),
\end{eqnarray}
where $ \psi(x,u,u')= \dfrac{e^x-1}{40}\left(u'^2-u-\frac{\cos x}{4}\right)$, $ \xi=0.2,~ \eta=0.3,~\lambda_1=\dfrac{1}{4} $ and $ \lambda_2=\dfrac{1}{9}. $ We consider initial lower and upper solutions $ c(x)=-1.905-\dfrac{x}{2}+\dfrac{x^2}{8},~d(x)=1.9+\dfrac{x}{2} $ respectively, where $ d(x) \geq c(x)$. Since $ \psi(x,u,u') $ is one sided Lipschitz in $ u $ so lipschitz constant is $ L_1=0.042957$. Also $f$ is Lipschitz in $ u' $ therefore we derive from $[A'_5]$, $ L_2(x)=\dfrac{2P(e^x-1)}{40} $, where $ P>0 $ such that  $ \parallel u'\parallel_\infty\leq P, ~~\forall x\in [0,1]$. Here we obtained $ \phi(|s|)=0.042957(|s^2|+2.65)$. By using $[A'_6]$ we obtain $ P=5.868826 $. The Range  $(-\alpha_0, -\beta_0)$ of $ k<0 $ is computed by using condition $[A'_1]$, $[A'_2]$ and  Mathematica.11.3, where $\beta_0=1.447171$ and $-\alpha_0$ is not sharp.
\begin{figure}[H]
   \begin{minipage}{0.48\textwidth}
     \centering
     \includegraphics[width=.7\linewidth]{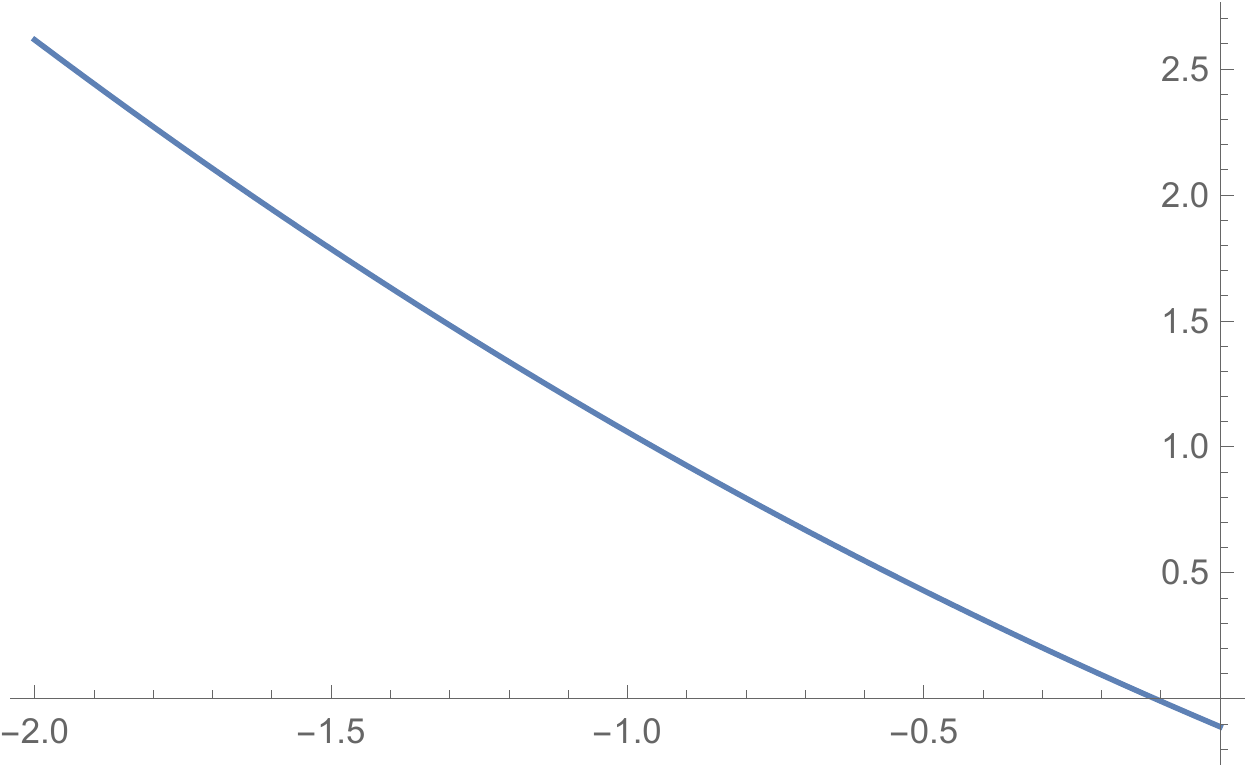}
     \caption{$\sqrt{\mid k\mid}\cosh\sqrt{\mid k\mid}-\lambda_2 \sinh\sqrt{\mid k\mid}$}\label{p2sec6_well_fig1}
   \end{minipage}\hfill
   \begin{minipage}{0.48\textwidth}
     \centering
     \includegraphics[width=.7\linewidth]{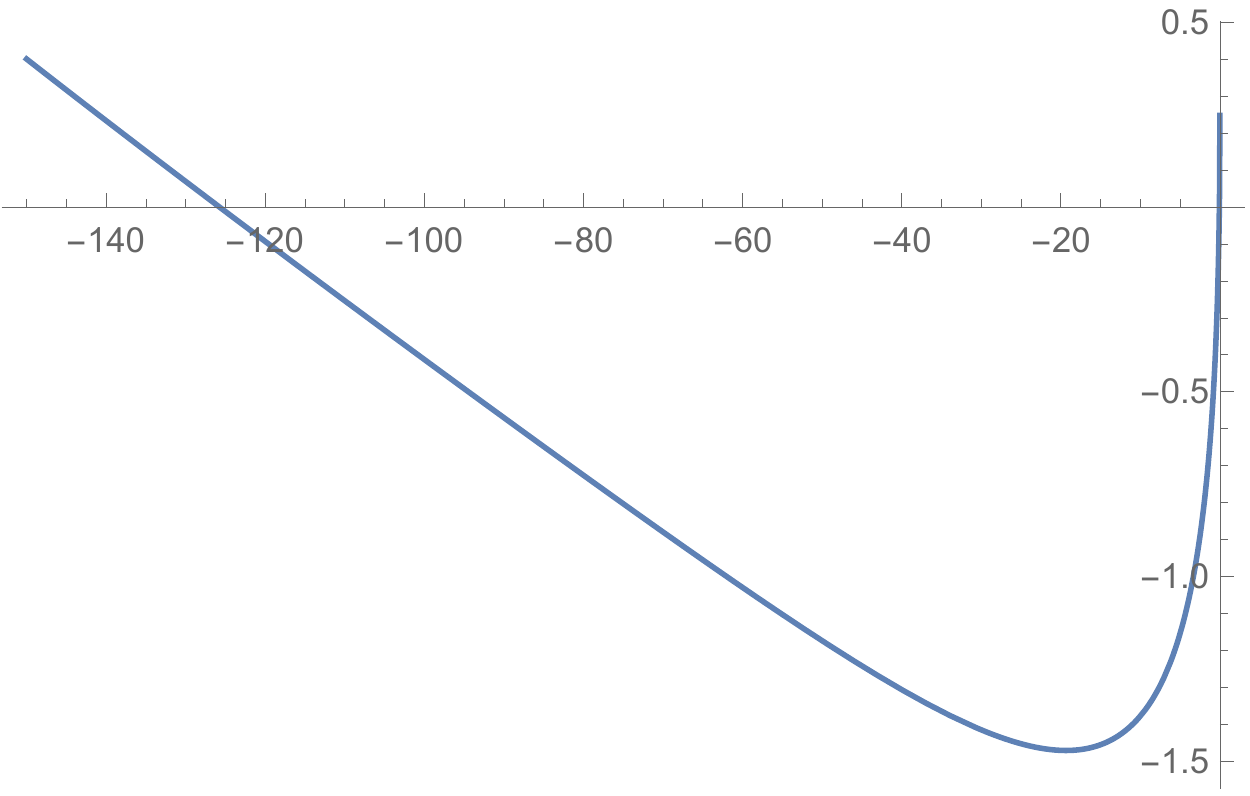}
     \caption{$\sqrt{\mid k\mid}\sinh\sqrt{\mid k\mid}\xi+(\lambda_1 -\sqrt{\mid k\mid})\cosh\sqrt{\mid k\mid}\xi$}\label{p2sec6_well_fig2}
   \end{minipage}
\end{figure}

\begin{figure}[H]
   \begin{minipage}{0.48\textwidth}
     \centering
     \includegraphics[width=.7\linewidth]{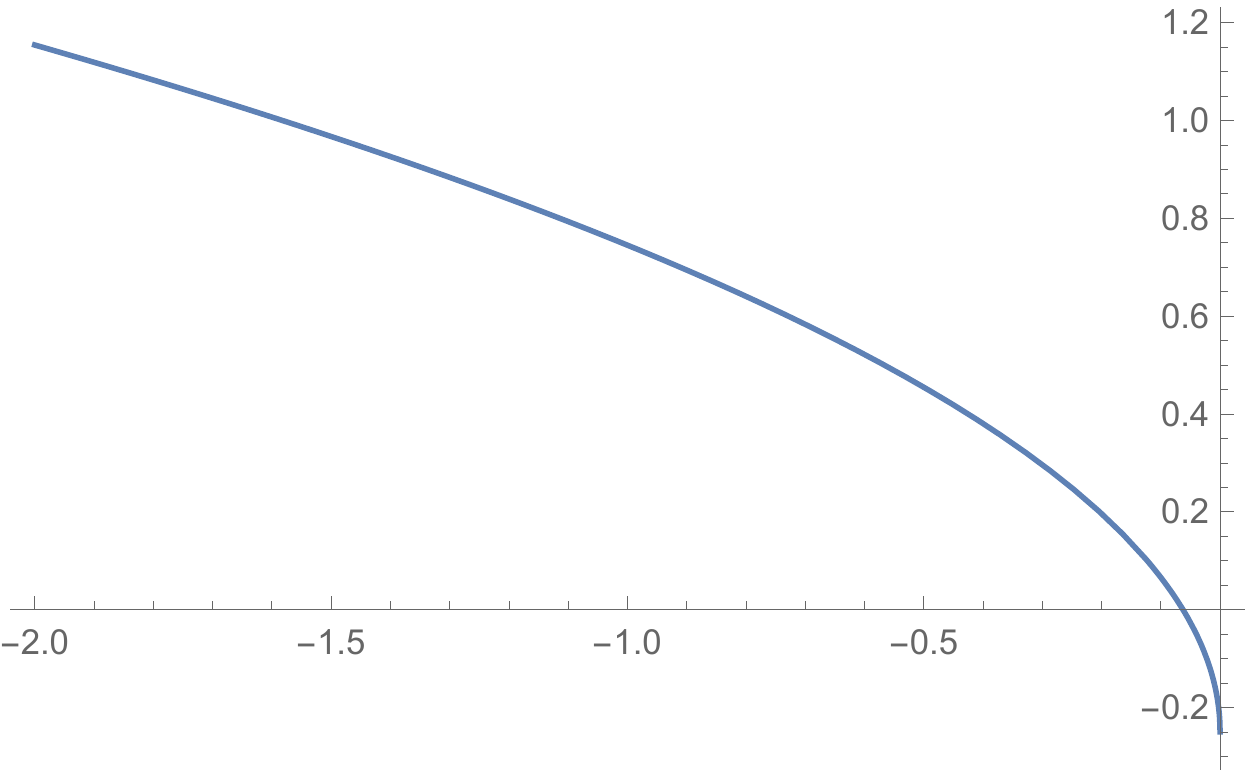}
     \caption{$\sqrt{\mid k\mid}-\lambda_1 \cosh\sqrt{\mid k\mid}\xi$}\label{p2sec6_well_fig3}
   \end{minipage}\hfill
   \begin{minipage}{0.48\textwidth}
     \centering
     \includegraphics[width=.7\linewidth]{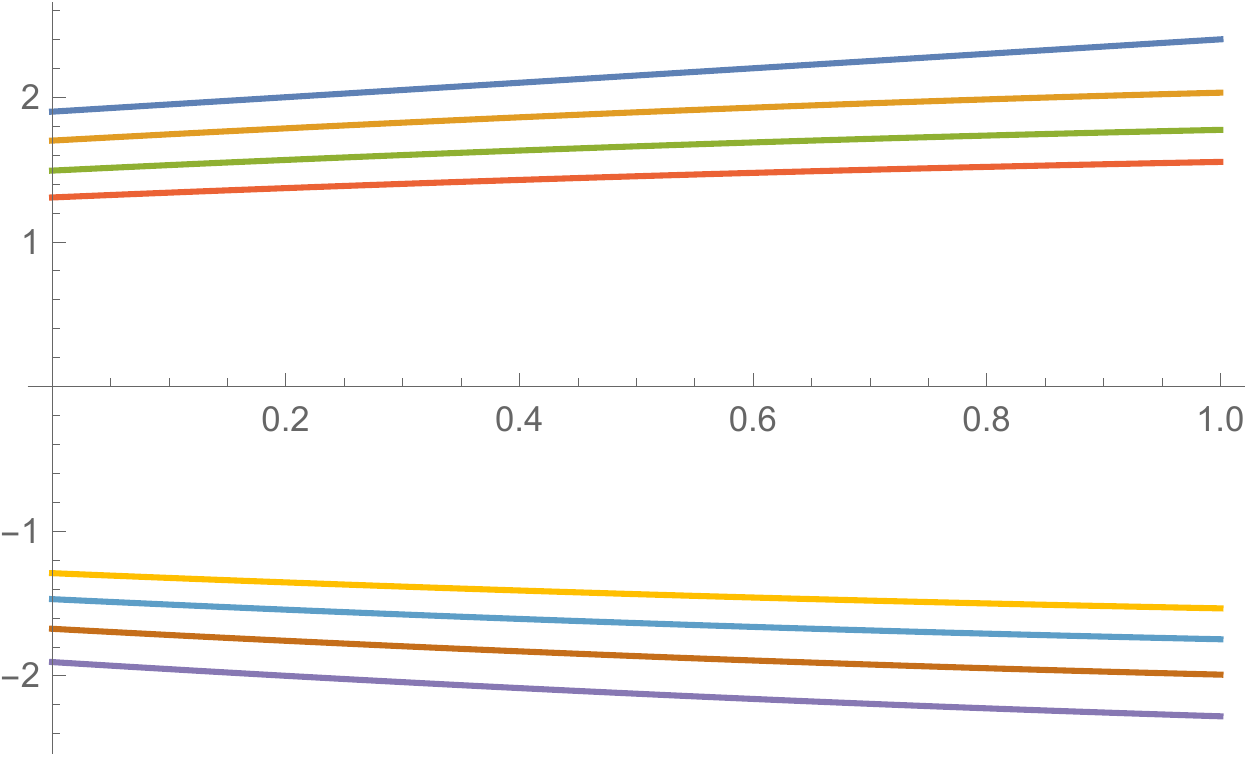}
     \caption{$k=-1.0698$}\label{p2sec6_well_fig4}
   \end{minipage}
\end{figure}

\begin{figure}[H]
   \begin{minipage}{0.48\textwidth}
     \centering
     \includegraphics[width=.7\linewidth]{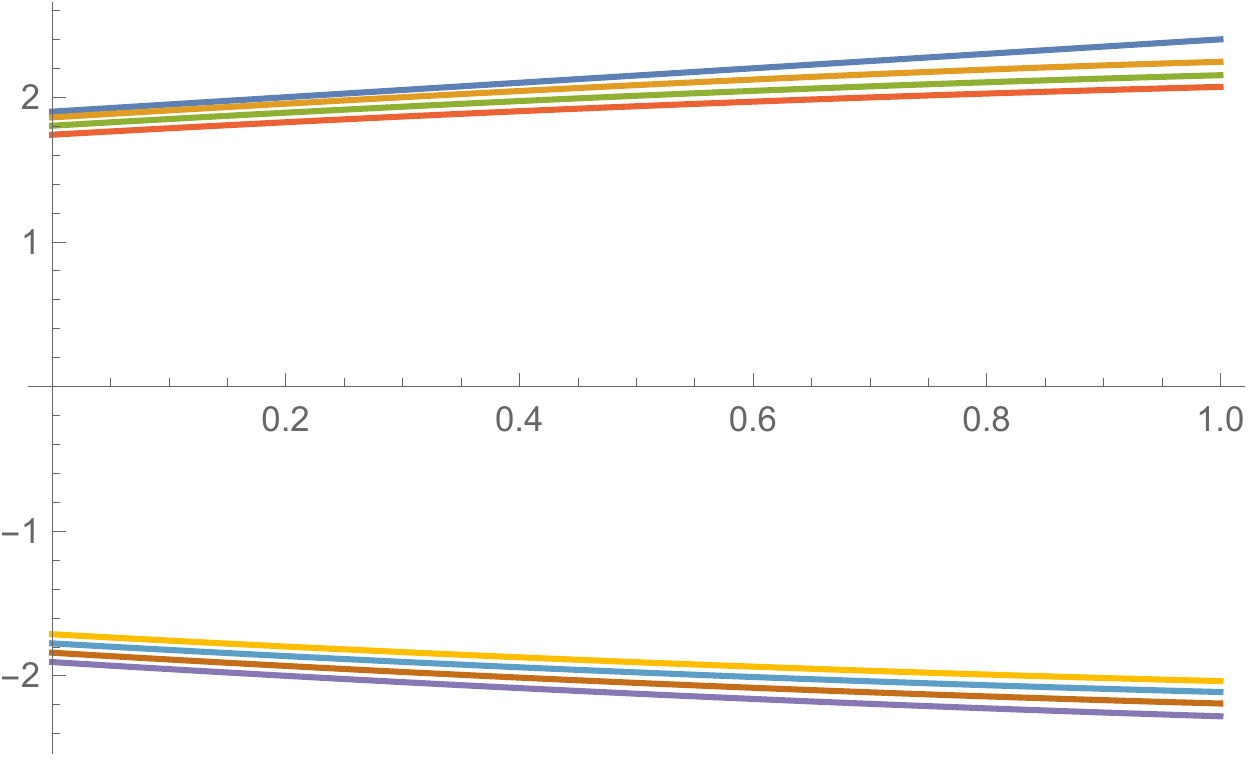}
     \caption{$k=-4$}\label{p2sec6_well_fig3a}
   \end{minipage}\hfill
   \begin{minipage}{0.48\textwidth}
     \centering
     \includegraphics[width=.7\linewidth]{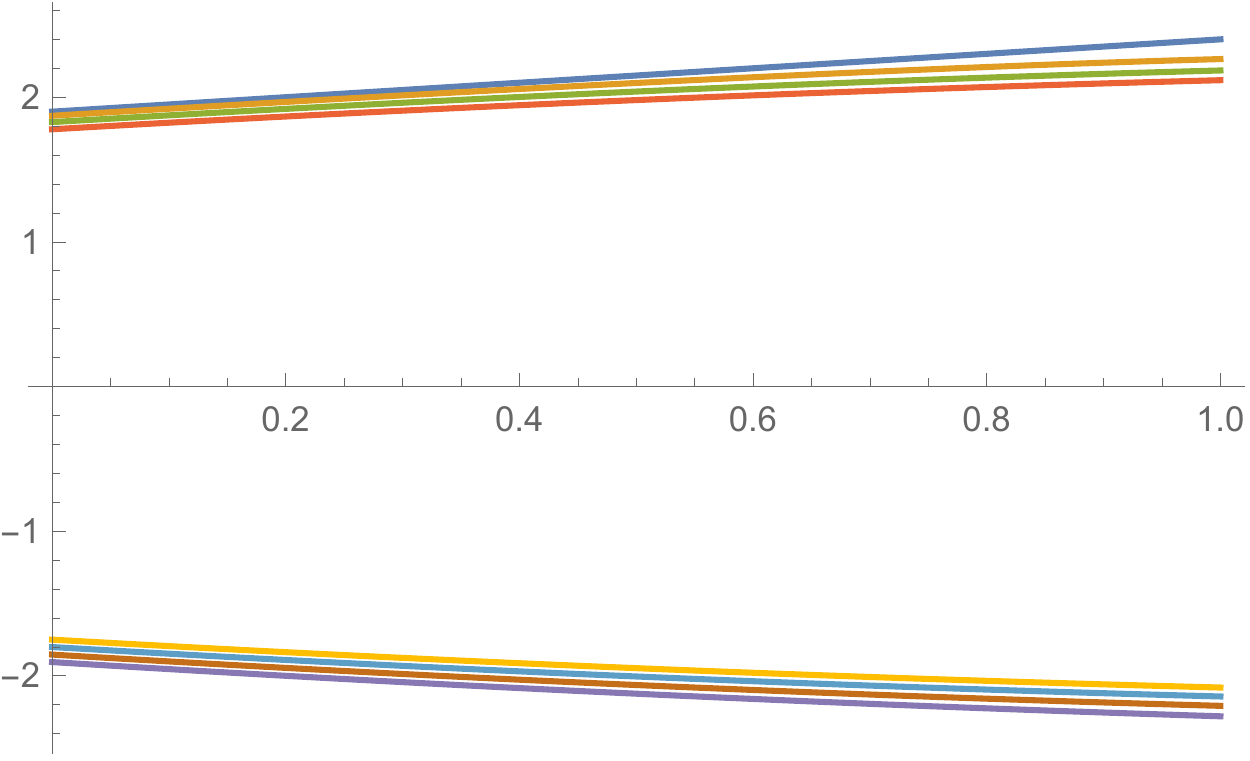}
     \caption{$k=-5$}\label{p2sec6_well_fig4a}
   \end{minipage}
\end{figure}
\textbf{Remark 4:} In figure 11 we observe that for $k=-1.0698$ the sequences of lower and upper solutions are monotone but we have obtained the range for $k$ is $(-\alpha_0, -\beta_0)$, where $\beta_0=1.447171$.
\section{Conclusion}
In this paper we have considered a non linear four point BVPs and scrutinized a technique which is called MI-Technique with upper-lower solutions. We have dealt with both cases reverse and well. We have defined both upper-lower solutions, maximum anti-maximum principle and sign of $G(x,s)$  for both positive and negative case. We observed that to prove proposition \ref{ppr2subsec3prop3} and \ref{ppr2sec4prop3} we need that $\psi(x,u,u')$ should be one Lipschitz in $u$ and Lipschitz in $u'$ with  Lipschitz functions $L_1$ and $L_2(x)$ respectively, where $L_2(x)$ is non-negative function of $x$. We have shown that if we are able to construct initial lower-upper solutions then the sequences of lower-upper solutions gives guarantee of the uniform convergence. For numerical  point of view this technique is easy to handle. We have illustrated two example for $k>0$ and $k<0$ and graphically we have shown that solution converges uniformly. For this we have used Mathematica-11.3. We have considered $k$ as a constant and not equal to zero. We have also obtained some range for both cases in which MI-Technique satisfy all the condition of our problem which we have deduced.
\bibliography{MasterS}
\bibliographystyle{plain}
\end{document}